\documentclass[a4paper]{amsart}
\usepackage{graphicx}
\usepackage{amssymb}
\usepackage{amsmath}
\usepackage{amsthm}
\usepackage{amscd}
\usepackage{color}
\usepackage{colordvi}
\usepackage[all,2cell]{xy}
\usepackage{picture}
\usepackage{multirow}
\usepackage{makecell}

\usepackage{amsfonts,enumerate,verbatim,mathtools,tikz,bm,tikz-cd,hyperref,comment}
\hypersetup{
	colorlinks=true,
	linkcolor=blue,
	filecolor=magenta,
	urlcolor=cyan,
}
\UseAllTwocells \SilentMatrices
\newtheorem{thm}{Theorem}[section]

\newtheorem{cor}[thm]{Corollary}
\newtheorem{lem}[thm]{Lemma}
\newtheorem{exm}[thm]{Example}

\newtheorem{prop}[thm]{Proposition}
\theoremstyle{definition}
\newtheorem{defn}[thm]{Definition}
\theoremstyle{remark}
\newtheorem{rem}[thm]{\bf Remark}
\numberwithin{equation}{section}

\begin{document}
\title[Skew group categories, Cartan matrices and folding]{Skew group categories, algebras associated to Cartan matrices and folding of root lattices}
\author[Xiao-Wu Chen, Ren Wang] {Xiao-Wu Chen, Ren Wang$^*$}

\thanks{$^*$ The corresponding author}
\thanks{}
\date{\today}

\thanks{E-mail: xwchen$\symbol{64}$mail.ustc.edu.cn; renw$\symbol{64}$mail.ustc.edu.cn}

\maketitle

\dedicatory{}%
\commby{}%
%\begin{center}
%\end{center}

\begin{abstract}
For a finite group action on a finite EI quiver, we construct its `orbifold' quotient EI quiver. The free EI category associated to the quotient EI quiver is equivalent to the skew group category with respect to the given group action. Specializing the result to a finite group action on a finite acyclic quiver, we prove that, under reasonable conditions, the skew group category of the path category is equivalent to a finite EI category of Cartan type. If the ground field is of characteristic $p$ and the acting group is a cyclic $p$-group, we prove that the skew group algebra of the path algebra is Morita equivalent to the algebra associated to a Cartan matrix, defined in [C. Geiss, B. Leclerc, and J. Schr\"{o}er, {\em Quivers with relations for symmetrizable Cartan matrices I: Foundations},  Invent. Math. {\bf 209} (2017), 61--158]. We apply the Morita equivalence to construct a categorification of the folding projection between the root lattices with respect to a graph automorphism. In the Dynkin cases,  the restriction of the categorification to indecomposable modules corresponds to the folding of positive roots.
 \end{abstract}

\section{Introduction}

\subsection{The background}

The folding of root lattices is classic \cite{St} and  plays a  significant role in Lie theory when getting from the simply-laced cases to the non-simply-laced cases. The starting point is the fact that a symmetrizable generalized Cartan matrix $C$ is determined by  a finite graph $\Gamma$  with an admissible automorphism $\sigma$ \cite{St, Lus}. There is a surjective homomorphism, called the \emph{folding projection},
$$\boldsymbol{f}\colon \mathbb{Z}\Gamma_0\longrightarrow \mathbb{Z}(\Gamma_0/{\langle \sigma \rangle})$$
from the root lattice of $\Gamma$ to that of $C$, which preserves simple roots; see \cite[Section~10.3]{Spr}. Here, $\Gamma_0$ denotes the set of vertices in $\Gamma$, and the orbit set $\Gamma_0/{\langle \sigma \rangle}$ indexes both the rows and columns of $C$, so that we identify $\mathbb{Z}(\Gamma_0/{\langle \sigma \rangle})$ with the root lattice of $C$. It is proved by \cite[Proposition~15]{Hu04} that the folding projection restricts to a surjective map
$$\boldsymbol{f}\colon \Phi(\Gamma)\longrightarrow \Phi(C)$$
between the root systems \cite{Kac}, known as the \emph{folding} of roots.

Let $\mathbb{K}$ be a field, and  $\Delta$ be a finite acyclic quiver such that its underlying graph is $\Gamma$. The path algebra $\mathbb{K}\Delta$ is finite dimensional and hereditary. It is well known that the category  of finite dimensional $\mathbb{K}\Delta$-modules, denoted by $\mathbb{K}\Delta\mbox{-mod}$,  categorifies the root lattice $\mathbb{Z}\Gamma_0$ in the following manner \cite{Gab}: the dimension vector ${\underline{\rm dim}}(M)$ of any $\mathbb{K}\Delta$-module $M$ belongs to $\mathbb{Z}\Gamma_0$, where simple $\mathbb{K}\Delta$-modules correspond to simple roots. Gabriel's theorem  \cite[1.2~Satz]{Gab}, one of the foundations in modern representation theory of algebras, states that if $\Delta$ is of Dynkin type, then indecomposable $\mathbb{K}\Delta$-modules correspond bijectively to positive roots in $\Phi(\Gamma)$.

Associated to a symmetrizable generalized Cartan matrix $C$, a finite dimensional $1$-Gorenstein algebra $H$ is defined in \cite{GLS}. The category of finite dimensional $\tau$-locally free $H$-modules, denoted by $H\mbox{-{\rm mod}}^{\tau \mbox{-{\rm lf}}}$,  categorifies  the root lattice $\mathbb{Z}(\Gamma_0/{\langle \sigma \rangle})$ in a similar manner: the rank vector $\underline{\rm rank}(X)$ of any $\tau$-locally free $H$-module $X$ belongs to $\mathbb{Z}(\Gamma_0/{\langle \sigma \rangle})$, where generalized simple $H$-modules correspond to simple roots.  \cite[Theorem~1.3]{GLS}, a remarkable analogue of Gabriel's theorem,  states that if $C$ is of Dynkin type, then indecomposable $\tau$-locally free $H$-modules correspond bijectively to positive roots  in $\Phi(C)$.

We mention that the categorification in \cite{GLS} works over an arbitrary ground field. In particular, it works for algebraically closed fields, and then certain geometric consideration for $\mathbb{K}\Delta$ carries over to $H$; see \cite{Gei}.  The traditional categorification of  $\mathbb{Z}(\Gamma_0/{\langle \sigma \rangle})$ for a non-symmetric Cartan matrix  uses  species \cite{DR}, where the ground field has to  be chosen suitably and can not be algebraically closed.

In view of the above work, the following question is natural and fundamental: how to categorify the folding projection $\boldsymbol{f}$ between the root lattices?  More precisely, is there an additive functor $\Theta\colon \mathbb{K}\Delta\mbox{-mod}\rightarrow H\mbox{-mod}^{\tau\mbox{\rm -lf}}$ making the following diagram
\[\xymatrix{
\mathbb{K}\Delta\mbox{-{\rm mod}} \ar[d]_-{\underline{\rm dim}} \ar@{.>}[rr]^-{\Theta} && H\mbox{-{\rm mod}}^{\tau \mbox{-{\rm lf}}} \ar[d]^-{\underline{\rm rank}}\\
\mathbb{Z}\Gamma_0 \ar[rr]^-{\boldsymbol{f}} && \mathbb{Z}(\Gamma_0/{\langle \sigma\rangle})
}
\]
 commute? Such a functor $\Theta$ might be called a \emph{categorification} of $\boldsymbol{f}$.

We will construct such a categorification under the assumptions that the characteristic  ${\rm char}(\mathbb{K})=p$ of the field is positive and that the automorphism $\sigma$ is of order $p^a$ for some $a\geq 1$. Moreover, if $\Delta$ is of Dynkin type, $\Theta$ preserves indecomposable modules and categorifies  the folding of positive roots.

For our purpose, it is very natural to require that $\sigma$ preserves the orientation, that is, it acts on  $\Delta$ by quiver automorphisms. We will work in a slightly more general setting, namely, finite group actions on finite free EI categories.

Recall that a finite category is  EI provided that each endomorphism is invertible; in particular, the endomorphism monoid of each object is a finite group. For example,  the path category of a finite acyclic quiver is EI. The study of finite EI categories goes back to \cite{Lvc},  and is used to reformulate and extend Alperin's weight conjecture \cite{We08, Lin14}. We mention that EI categories are very similar to graphs of groups in the sense of Bass-Serre \cite{Bass, Serre}.

As an EI analogue of a path category, the notion of a finite free EI category is introduced in \cite{L2011}. We are mostly interested in EI categories of Cartan type \cite{CW},  which are certain finite free EI categories associated to symmetrizable generalized Cartan matrices. The construction of  the categorification $\Theta$ relies  on the isomorphism \cite{CW} between the category algebra of an EI category of Cartan type and the algebra $H$  in \cite{GLS}.

\subsection{The main results}
Let $\mathcal{C}$ be a finite category and $G$ be a finite group. Assume that $G$ acts on $\mathcal{C}$ by categorical automorphisms. As a very special case of the Grothendieck construction, we have the \emph{skew group category} $\mathcal{C}\rtimes G$. The terminology is justified by the following fact: the category algebra $\mathbb{K}(\mathcal{C}\rtimes G)$ is isomorphic to $\mathbb{K}\mathcal{C}\#G$, the skew group algebra of the category algebra $\mathbb{K}\mathcal{C}$ with respect to the induced $G$-action.

Following \cite[Definition~2.1]{L2011}, a \emph{finite EI quiver} $(Q, U)$ consists of a finite acyclic quiver $Q$ and an assignment $U$ on $Q$. The assignment $U$ assigns to each vertex $i$ of $Q$ a finite group $U(i)$, and to each arrow $\alpha$, a finite  $(U(t\alpha), U(s\alpha))$-biset $U(\alpha)$. Here, $t\alpha$ and $s\alpha$ denote the terminating vertex and starting vertex of $\alpha$, respectively.

In a natural manner, each finite EI quiver $(Q, U)$ gives rise to a finite EI category $\mathcal{C}(Q, U)$ such that the objects of $\mathcal{C}(Q, U)$ are precisely the vertices of $Q$, the automorphism group of $i$ coincides with $U(i)$, and that elements of $U(\alpha)$ correspond to unfactorizable morphisms. By \cite[Definition~2.2 and Proposition~2.8]{L2011}, a finite EI category $\mathcal{C}$ is said to be \emph{free}, provided that it is equivalent to $\mathcal{C}(Q, U)$ for some finite EI quiver $(Q, U)$.

Let $G$ be a finite group acting on $(Q, U)$ by EI quiver automorphisms. Then $G$ acts naturally on the EI category $\mathcal{C}(Q, U)$. We form the skew group category $\mathcal{C}(Q, U)\rtimes G$.  Inspired by \cite[Section~3]{Bass}, we construct  the `orbifold' \emph{quotient} EI quiver $(\overline{Q},\overline{U})$. Here, $\overline{Q}$ is the quotient quiver $Q$ by $G$, and the construction of the assignment $\overline{U}$ is quite involved. We mention that for each vertex $\boldsymbol{i}$ of $\overline{Q}$, the finite group $\overline{U}(\boldsymbol{i})$ is a semi-direct product of $U(i)$ with the stabilizer $G_i$ for some vertex $i$ of $Q$. For details, we refer to Subsection~\ref{subsec:quo}.

The first main result identifies  the category associated to the quotient EI quiver with the skew group category, and thus  justifies the `orbifold' quotient construction.
\vskip 5pt

\noindent {\bf Theorem A.}\; \emph{Let $(Q, U)$ be a finite EI quiver with a $G$-action, and $(\overline{Q},\overline{U})$ be its quotient EI quiver. Then there is an equivalence of categories
$$\mathcal{C}(\overline{Q},\overline{U})\simeq \mathcal{C}(Q, U)\rtimes G.$$
}

We mention that Theorem~A (= Theorem~\ref{thm:equiv-cat}) might be viewed as a combinatorial analogue to the well-known fact: the skew group algebra of a commutative algebra with respect to a finite group action is closely related to the corresponding quotient singularity; for example, see \cite{Wem}.

 Let $\Delta$ be a finite  acyclic quiver. Denote by $\mathcal{P}_\Delta$ its path category. We view $\Delta$ as a finite EI quiver $(\Delta,U_{\rm tr})$ with trivial assignment $U_{\rm tr}$. Then we have
 $$\mathcal{C}(\Delta, U_{\rm tr})=\mathcal{P}_\Delta.$$
 Assume that $G$ acts on $\Delta$ by quiver automorphisms. It induces a $G$-action on $(\Delta,U_{\rm tr})$. Denote by  $(\overline{\Delta},\overline{U}_{\rm tr})$ the corresponding quotient EI quiver, where $\overline{\Delta}$ is the quotient quiver $\Delta$ by $G$. Theorem~A implies that there is an equivalence of categories
 \begin{align}\label{equ:equi-intro}
 \mathcal{C}(\overline{\Delta},\overline{U}_{\rm tr}) \simeq \mathcal{P}_{\Delta}\rtimes G.
 \end{align}

By a \emph{Cartan triple} $(C,D,\Omega)$, we mean that $C$ is a symmetrizable generalized  Cartan matrix, $D$  is its symmetrizer and that $\Omega$ is an acyclic orientation of $C$. Following \cite[Section~1.4]{GLS}, we denote by $H(C, D, \Omega)$ the $1$-Gorenstein $\mathbb{K}$-algebra associated to any Cartan triple $(C, D, \Omega)$. Similarly, we associate a finite free EI category $\mathcal{C}(C,D,\Omega)$, called an \emph{EI category of Cartan type}, to any Cartan triple $(C, D, \Omega)$; see \cite[Definition~4.1]{CW}.

 As is well known, there is a Cartan triple $(C,D,\Omega)$  associated to the above $G$-action on $\Delta$ such that both the rows and columns of $C$ and $D$ are indexed by the orbit set $\overline{\Delta}_0=\Delta_0/G$. Here, $\Delta_0$ denotes the set of vertices in $\Delta$. Moreover, for each $G$-orbit $\boldsymbol{i}$ of vertices, the corresponding diagonal entry of $D$ is $|G|/|\boldsymbol{i}|$; the corresponding off-diagonal entry of $C$ is $$c_{\boldsymbol{i},\boldsymbol{j}}=-\frac{N_{\boldsymbol{i},\boldsymbol{j}}}{|\boldsymbol{j}|},$$
 where $|\boldsymbol{i}|$ denotes the cardinality of the $G$-orbit $\boldsymbol{i}$ and $N_{\boldsymbol{i},\boldsymbol{j}}$ denotes the number of arrows in $\Delta$ between the $G$-orbit $\boldsymbol{i}$ and $G$-orbit $\boldsymbol{j}$. The orientation of $\Omega$ is induced from the one of $\Delta$.

The second main theorem establishes an equivalence between the skew group category and the EI category of Cartan type. Based on \cite{CW}, we obtain a Morita equivalence between the skew group algebra $\mathbb{K}\Delta\#G$ and $H(C, D, \Omega)$.

\vskip 5pt

\noindent {\bf Theorem B.}\;
\emph{Let $\Delta$ be a finite acyclic quiver with a $G$-action that satisfies $(\dagger 1)\mbox{-}(\dagger 3)$ in Subsection~\ref{subset:qC}. Assume that $(C,D,\Omega)$ is the associated Cartan triple. Then we have the following statements.
\begin{enumerate}
	\item There is an equivalence of categories
	$$\mathcal{P}_{\Delta}\rtimes G\simeq \mathcal{C}(C,D,\Omega).$$
	\item Assume that ${\rm char}(\mathbb{K})=p>0$ and that $G$ is a $p$-group. Then the skew group algebras $\mathbb{K}\Delta\#G$ and $H(C,D,\Omega)$ are Morita equivalent.
\end{enumerate}
}
\vskip3pt

The above technical conditions $(\dagger 1)\mbox{-}(\dagger 3)$ are easily satisfied when $G$ is cyclic. On the other hand, examples where they do hold seem to be ubiquitous; see Example~\ref{exm:ubi}. In view of (\ref{equ:equi-intro}), the core of the proof of  Theorem~B is to describe the assignment $\overline{U}_{\rm tr}$ in the quotient EI quiver. We refer to Theorem~\ref{thm:quiver} for more details.

The equivalence  and the Morita equivalence in Theorem~B indicate that both EI categories of Cartan type \cite{CW} and the algebra $H(C,D, \Omega)$ \cite{GLS} arise naturally in the representation theory of quivers with automorphisms \cite{Lus, Hu04}.

The Morita equivalence in Theorem~B(2) yields an equivalence between module categories
$$\Psi \colon \mathbb{K}\Delta\#G\text{-mod}\overset{\sim}{\longrightarrow} H(C,D,\Omega)\text{-{\rm mod}}.$$
We have the obvious \emph{induction functor}
$$-\#G\colon \mathbb{K}\Delta\text{-mod}\longrightarrow \mathbb{K}\Delta\#G\text{-mod}, \quad M\mapsto M\#G.$$

For $\tau$-locally free modules over $H=H(C, D, \Omega)$, we refer to  \cite[Definition~1.1 and Section~11]{GLS}. Denote by $H\mbox{-mod}^{\tau\mbox{-lf}}$ the full subcategory of $H\mbox{-mod}$ consisting of $\tau$-locally free modules.  In contrast to \cite{GLS}, we  do not require $\tau$-locally free $H$-modules to be indecomposable.

Recall that  $\mathbb{Z}\Delta_0$ and $\mathbb{Z}(\Delta_0/G)$ denote the root lattices of $\Delta$ and $C$, respectively. The sets of positive roots are denoted by $\Phi^+(\Delta)$ and $\Phi^+(C)$, respectively.

The third main result shows that the composite functor $\Psi\circ (-\#G)$ is the pursued categorification of the folding projection $\boldsymbol{f}$; see  Theorem~\ref{thm:categorify} and Proposition~\ref{prop:Dynkin}.

\vskip 5pt

\noindent {\bf Theorem~C.}\;  \emph{Assume that ${\rm char}(\mathbb{K})=p>0$ and that $G$ is a cyclic $p$-group. Assume that $G$ acts on a finite acyclic quiver $\Delta$ such that $G_{\alpha}=G_{s(\alpha)}\cap G_{t(\alpha)}$ for each arrow $\alpha$ in $\Delta$. Assume that $(C,D,\Omega)$ is its associated Cartan triple. Then we have the following commutative diagram.
\begin{equation*}
\xymatrix@C=1.8cm{
	\mathbb{K}\Delta\text{-{\rm mod}}\ar[r]^-{\Psi\circ (-\#G)\; \; \;} \ar[d]_{\underline{\rm dim}} &  H(C,D,\Omega)\text{-{\rm mod}}^{\tau\text{-{\rm lf}}} \ar[d]^{\underline{\rm rank}} \\
	\mathbb{Z}\Delta_0\ar[r]^{\boldsymbol{f}} & \mathbb{Z}(\Delta_0/G)
}
\end{equation*}}

\emph{Assume further  that $\Delta$ is of Dynkin type. Then the above commutative diagram restricts to the following one.
\begin{equation}\label{diag:catf}
\xymatrix@C=1.8cm{
	\mathbb{K}\Delta\text{-{\rm ind}}\ar[r]^-{\Psi\circ (-\#G)\; \; \;} \ar[d]_{\underline{\rm dim}} &  H(C,D,\Omega)\text{-{\rm ind}}^{\tau\text{-{\rm lf}}} \ar[d]^{\underline{\rm rank}} \\
	\Phi^+(\Delta)\ar[r]^{\boldsymbol{f}} & \Phi^+(C)
}
\end{equation}
}

Here, $G_\alpha$, $G_{s(\alpha)}$ and $G_{t(\alpha)}$ denote the stabilizers of an arrow $\alpha$, its starting vertex $s(\alpha)$ and terminating vertex $t(\alpha)$, respectively. The natural condition $G_{\alpha}=G_{s(\alpha)}\cap G_{t(\alpha)}$ implies that  the technical conditions $(\dagger 1)\mbox{-}(\dagger 3)$ in Theorem~B hold.

In (\ref{diag:catf}), we denote by $\mathbb{K}\Delta\mbox{-ind}$ a complete set of representatives of indecomposable $\mathbb{K}\Delta$-modules. Similarly, $H(C,D,\Omega)\text{-{\rm ind}}^{\tau\text{-{\rm lf}}}$ is a complete set of representatives of indecomposable $\tau$-locally free $H(C, D, \Omega)$-modules.

 In the Dynkin cases, by \cite[1.2~Satz]{Gab} and  \cite[Theorem~1.3]{GLS}, the vertical arrows in (\ref{diag:catf}) are both bijections. Since $\boldsymbol{f}\colon \Phi^+(\Delta)\rightarrow \Phi^+(C)$ is surjective, we infer that up to the equivalence $\Psi$, every $\tau$-locally free $H(C, D, \Omega)$-module is induced from $\mathbb{K}\Delta\mbox{-ind}$. This yields a new interpretation of those $H(C, D, \Omega)$-modules \cite{GLS} that categorify the root system $\Phi^+(C)$.

In view of \cite[Section~14.1]{Lus} and \cite{Hu04}, it has been expected that skew group algebras play a role in categorifying the root lattice for symmetrizable generalized  Cartan matrices. We observe that in \cite[Section~4]{Hu04} the characteristic of the ground field is assumed to be coprime to the order of the acting group. In contrast, the feature of Theorem~C is the assumptions that the ground field $\mathbb{K}$ is  of characteristic $p$ and that the order of the acting group $G$ is a $p$-power.

\subsection{The structure} The paper is structured as follows. In Section~2, we prove that for a finite group action on a finite category, the skew group category is EI if and only if so is the given category; see Proposition~\ref{prop:skew-EI}. In Section~3, we study the unique factorization property of morphisms and free EI categories. We prove that for a finite group action on a finite category, the skew group category is free EI if and only if so is the given category; see Proposition~\ref{prop:Free}. In Section~4, we recall finite  EI quivers introduced in \cite{L2011}, and prove a universal property of the free EI category associated to a finite EI quiver; see Proposition~\ref{prop:universal}.

For a finite group action on a finite EI quiver, we construct its `orbifold' quotient EI quiver explicitly in Section~5. Theorem~\ref{thm:equiv-cat} states that the category associated to the quotient EI quiver is equivalent to the skew group category.

In Section~6, we recall the algebras $H$ \cite{GLS} and EI categories \cite{CW} associated to Cartan triples. For a finite group action on a finite acyclic quiver, we give sufficient conditions on when the quotient EI quiver is of Cartan type. Consequently, the skew group algebra of the path algebra is Morita equivalent to the algebra $H$; see Theorem~\ref{thm:quiver}.

In the final section, we first study induced modules over skew group algebras. We apply Theorem~\ref{thm:quiver} to the case where a finite cyclic $p$-group acts on a finite acyclic quiver. Theorem~\ref{thm:categorify} obtains a categorification of the folding projection $\boldsymbol{f}$, namely an additive functor from the module category over the path algebra to the category of $\tau$-locally free $H$-modules. In the Dynkin cases, restricting the categorification  to indecomposable modules, we obtain a categorification of the folding of positive roots; see Proposition~\ref{prop:Dynkin}. In the end, we give an explicit example to illustrate the categorification.

By default, a module means a finite dimensional left module. For a finite dimensional algebra $A$, we denote by $A\mbox{-mod}$ the abelian category of finite dimensional left $A$-modules. We use ${\rm rad}(A)$ to denote the Jacobson radical of $A$. The unadorned tensor $\otimes$ means the tensor product over the ground field $\mathbb{K}$.

\section{Skew group categories}

In this section, we recall basic facts about finite group actions on finite categories. The EI property of a skew group category is studied in Proposition~\ref{prop:skew-EI}.

\subsection{Finite $G$-categories}

Let $\mathcal{C}$ be a finite category, that is, a category with only finitely many morphisms. As any object is determined by its identity endomorphism, the finite category $\mathcal{C}$  necessarily has only finitely many objects. Denote by ${\rm Obj}(\mathcal{C})$ (\emph{resp.} ${\rm Mor}(\mathcal{C})$) the finite set of objects (\emph{resp.} morphisms) in $\mathcal{C}$. We denote by ${\rm Aut}(\mathcal{C})$ the automorphism group of $\mathcal{C}$.

Let $G$ be a finite group with its unit $1_G$. A finite \emph{$G$-category} $\mathcal{C}$  is a finite category equipped with a group homomorphism
$$\rho\colon G\longrightarrow {\rm Aut}(\mathcal{C}).$$
To simplify the notation, the following convention will be used: for $g\in G$ and $x\in {\rm Obj}(\mathcal{C})$, we write $g(x)=\rho(g)(x)$; for $\alpha\in {\rm Mor}(\mathcal{C})$, we write $g(\alpha)=\rho(g)(\alpha)$.

For a finite $G$-category $\mathcal{C}$, we  will recall the \emph{skew group category} $\mathcal{C}\rtimes G$; compare \cite[Subsection~3.1]{RR} and \cite[Definition~2.3]{CiM}. It has the same objects as $\mathcal{C}$; for two objects $x$ and  $y$, the corresponding Hom set is defined to be
\[{\rm Hom}_{\mathcal{C}\rtimes G}(x,y)=\{(\alpha,g) \mid g\in G, \alpha\in {\rm Hom}_{\mathcal{C}}(g(x),y)\}.\]
For any morphisms $(\alpha,g)\in {\rm Hom}_{\mathcal{C}\rtimes G}(x,y)$ and $(\beta,h)\in {\rm Hom}_{\mathcal{C}\rtimes G}(y,z)$, the composition is  defined by
\begin{align}\label{equ:composition}
(\beta,h)\circ (\alpha,g)=(\beta\circ h(\alpha), hg).
\end{align}
We observe that the identity endomorphism of $x$ in $\mathcal{C}\rtimes G$ is given by $({\rm Id}_x, 1_G)$, where ${\rm Id}_x$ is the identity  endomorphism of $x$ in $\mathcal{C}$. We mention that the formation of a skew group category might be viewed as a very special case of the Grothendieck construction; compare  \cite[VI.8]{Gro} and \cite[Section~7]{We07}.

Let $\mathbb{K}$ be a field and $\mathcal{C}$ be a finite category. The \emph{category algebra} $\mathbb{K}\mathcal{C}$ of $\mathcal{C}$ is a finite dimensional $\mathbb{K}$-algebra  defined as
follows. As a $\mathbb{K}$-vector space, $\mathbb{K}\mathcal{C}=\bigoplus\limits_{\alpha \in {\rm Mor}(\mathcal{C})} \mathbb{K}\alpha$, and the product between the basis elements  is given by the following rule:
\[\alpha  \beta=\left\{\begin{array}{ll}
\alpha\circ\beta, & \text{ if }\text{$\alpha$ and $\beta$ can be composed in $\mathcal{C}$}; \\
0, & \text{otherwise.}
\end{array}\right.\]
The unit of $\mathbb{K}\mathcal{C}$ is given by $1_{\mathbb{K}\mathcal{C}}=\sum\limits_{x \in {\rm Obj}(\mathcal{C})}{\rm Id}_x$.

Denote by $(\mathbb{K}\mbox{-mod})^\mathcal{C}$ the category of covariant functors from $\mathcal{C}$ to $\mathbb{K}\mbox{-mod}$. There is a canonical equivalence
\begin{align}\label{equ:can-C}
{\rm can}\colon \mathbb{K}\mathcal{C}\mbox{-mod} \stackrel{\sim}\longrightarrow (\mathbb{K}\mbox{-mod})^\mathcal{C},
\end{align}
sending a $\mathbb{K}\mathcal{C}$-module $M$ to the functor ${\rm can}(M)\colon \mathcal{C}\rightarrow \mathbb{K}\mbox{-mod}$ described as follows: ${\rm can}(M)(x)={\rm Id}_x.M$ for each object $x$ in $\mathcal{C}$; for any morphism $\alpha\colon x\rightarrow y$, we have
$${\rm can}(M)(\alpha)\colon {\rm can}(M)(x)\longrightarrow {\rm can}(M)(y),\; m\mapsto \alpha.m.$$
For details, we refer to \cite[Proposition~2.1]{We07}.

Denote by ${\rm Aut}(\mathbb{K}\mathcal{C})$ the group of algebra automorphisms on $\mathbb{K}\mathcal{C}$. Each categorical automorphism on $\mathcal{C}$ induces uniquely an algebra automorphism on $\mathbb{K}\mathcal{C}$. Therefore, there is a canonical embedding of groups
$${\rm Aut}(\mathcal{C})\hookrightarrow {\rm Aut}(\mathbb{K}\mathcal{C}).$$

Assume that $\mathcal{C}$ is a finite $G$-category. The group homomorphism $\rho\colon G\rightarrow {\rm Aut}(\mathcal{C})$ induces a group homomorphism $\rho'\colon G\rightarrow {\rm Aut}(\mathbb{K}\mathcal{C})$. In other words, the group $G$ acts on the algebra $\mathbb{K}\mathcal{C}$ by algebra automorphisms. We denote by $\mathbb{K}\mathcal{C}\# G$ the corresponding \emph{skew group algebra}. Here, we recall that $\mathbb{K}\mathcal{C}\# G=\mathbb{K}\mathcal{C}\otimes \mathbb{K}G$ as a $\mathbb{K}$-vector space, where the tensor product $\alpha\otimes g$ is written as $\alpha\#g$. The multiplication is  given by
\[(\beta\# h)(\alpha \# g)=\beta h(\alpha)\# hg\]
for any $\alpha, \beta\in {\rm Mor}(\mathcal{C})$ and $g, h\in G$. We emphasize that on the right hand side, $\beta h(\alpha)$ means the product of $\beta$ and $h(\alpha)$ in $\mathbb{K}\mathcal{C}$, namely, the composition $\beta\circ h(\alpha)$ in $\mathcal{C}$.

The following easy observation, extending \cite[Lemma~2.3.2]{Xu},  justifies the terminology `skew group category'.

\begin{prop}\label{prop:skew}
Let $\mathcal{C}$ be a finite $G$-category. Then there is an isomorphism of algebras
$$\mathbb{K}(\mathcal{C}\rtimes G)\overset{\sim}{\longrightarrow} \mathbb{K}\mathcal{C}\# G, $$
sending a morphism $(\alpha, g)$ in $\mathcal{C}\rtimes G$ to the element $\alpha\# g$ in $\mathbb{K}\mathcal{C}\#G$. \hfill $\square$
\end{prop}

In the following lemma, we collect elementary facts on skew group categories.

\begin{lem}\label{lem:facts}
Let $\mathcal{C}$ be a finite $G$-category. Then the following two statements hold.
\begin{enumerate}
\item A morphism $(\alpha,g)$ in $\mathcal{C}\rtimes G$  is an isomorphism  if and only if $\alpha$ is an isomorphism in $\mathcal{C}$.
    \item For two objects $x$ and $y$ in $\mathcal{C}$, they are isomorphic in $\mathcal{C}\rtimes G$ if and only if $x$ is isomorphic to $g(y)$ in $\mathcal{C}$ for some $g\in G$.
\end{enumerate}
\end{lem}

\begin{proof} (1) For the ``if" part, we assume that $\alpha^{-1}$ is the inverse of $\alpha$ in $\mathcal{C}$. Then $(g^{-1}(\alpha^{-1}), g^{-1})$ is a well-defined morphism in $\mathcal{C}\rtimes G$; moreover, it is the required inverse of $(\alpha, g)$.

For the ``only if" part, we observe that the inverse of $(\alpha, g)$ has to be of the form $(\beta, g^{-1})$. Then it is direct to see that $g(\beta)$ is the inverse of $\alpha$, as required.

(2)  For the ``if" part, we assume that $\alpha\colon g(y) \rightarrow x$ is an isomorphism in $\mathcal{C}$. Then $(\alpha, g)$ is a morphism from $y$ to $x$ in $\mathcal{C}\rtimes G$; moreover, by (1) it is an isomorphism between $y$ and $x$.

For the ``only if" part, we assume that $(\alpha, g)\in {\rm Hom}_{\mathcal{C}\rtimes G}(y, x)$ is an isomorphism. By (1), we deduce that $\alpha$ is an isomorphism from $g(y)$ to $x$ in $\mathcal{C}$. \end{proof}

\subsection{The EI property}

Let $\mathcal{C}$ be a finite $G$-category as  above. For each object $x$ in $\mathcal{C}$, we denote by $G_x=\{g\in G \mid g(x)=x\}$ its stabilizer. We observe that $G_x$ acts on the monoid ${\rm Hom}_\mathcal{C}(x, x)$ by monoid automorphisms. Denote by ${\rm Hom}_\mathcal{C}(x, x)\rtimes G_x$ the corresponding semi-direct product. There is an inclusion between monoids
\begin{align*}
{\rm inc}_x\colon {\rm Hom}_\mathcal{C}(x, x)\rtimes G_x\hookrightarrow {\rm Hom}_{\mathcal{C}\rtimes G}(x, x), \quad (\alpha, g)\mapsto (\alpha, g).
\end{align*}

The following terminology is inspired by \cite[Subsection~12.1.1]{Lus}.

\begin{defn}\label{defn:adm}
A finite $G$-category $\mathcal{C}$ is \emph{admissible}, provided that for any  $x\in {\rm Obj}(\mathcal{C})$ and $g\in G$, ${\rm Hom}_{\mathcal{C}}(g(x),x)=\emptyset$ whenever $g(x)\neq x$. \hfill $\square$
\end{defn}

\begin{lem}\label{lem:adm}
A finite $G$-category $\mathcal{C}$ is admissible if and only if ${\rm inc}_x$ is surjective for each object $x$ in $\mathcal{C}$.
\end{lem}

\begin{proof}
The inclusion ${\rm inc}_x$ is not surjective if and only if there exists $g\in G$ satisfying $g(x)\neq x$ and ${\rm Hom}_\mathcal{C}(g(x), x)\neq \emptyset$. Then the result follows immediately.
\end{proof}

Recall from \cite{We07} that a finite category $\mathcal{C}$ is EI if every \underline{e}ndomorphism is an \underline{i}somorphism. Therefore, for each object $x$,  ${\rm Hom}_{\mathcal{C}}(x,x)={\rm Aut}_{\mathcal{C}}(x)$ is a finite group. Finite EI categories are of interest from many different perspectives; for example, see \cite{We08,Xu}.

\begin{prop}\label{prop:skew-EI}
Let $\mathcal{C}$ be a finite  $G$-category. Then $\mathcal{C}$ is an EI category if and only if so is $\mathcal{C}\rtimes G$.
\end{prop}

\begin{proof}
For the ``if" part, we assume that $\mathcal{C}\rtimes G$ is an EI category. For any $\alpha\in {\rm Hom}_{\mathcal{C}}(x,x)$,  $(\alpha,1_G)$ is an endomorphism of $x$ in $\mathcal{C}\rtimes G$. Since $\mathcal{C}\rtimes G$ is EI, $(\alpha, 1_G)$ is an isomorphism. By Lemma~\ref{lem:facts}(1), the endomorphism $\alpha$ is an isomorphism in $\mathcal{C}$, as required.

For the ``only if" part, we assume that $\mathcal{C}$ is an EI category. Any endomorphism of $x$ in $\mathcal{C}\rtimes G$ is of the form $(\alpha, g)$, where $\alpha\colon g(x)\rightarrow x$ is a morphism in $\mathcal{C}$. Assume that $g^d=1_G$ for some $d\geq 1$. Then we have a chain
$$x=g^d(x)\stackrel{g^{d-1}(\alpha)} \longrightarrow g^{d-1}(x) \longrightarrow \cdots \longrightarrow g^2(x)\stackrel{g(\alpha)}\longrightarrow g(x)\stackrel{\alpha}\longrightarrow x$$
of morphisms in $\mathcal{C}$. Since $\mathcal{C}$ is EI, it follows that all the morphisms in the chain are isomorphisms. In particular, the morphism $\alpha$ is  an isomorphism. Applying Lemma~\ref{lem:facts}(1), we infer that the endomorphism  $(\alpha, g)$ is an isomorphism, proving that $\mathcal{C}\rtimes G$ is an EI category.
\end{proof}

The following corollary follows immediately from Lemma~\ref{lem:adm}.

\begin{cor}\label{cor:EI}
Let $\mathcal{C}$ be a finite  admissible $G$-category. Assume that $\mathcal{C}$ is EI. Then for each object $x$, we have an identification of groups
$$\hskip 115pt
{\rm Aut}_\mathcal{C}(x)\rtimes G_x={\rm Aut}_{\mathcal{C}\rtimes G}(x).
\hskip 115pt \square$$
\end{cor}

\section{Free EI categories}

In this section, we study the unique factorization property of morphisms and free EI categories \cite{L2011}. We prove that a skew group category is free EI if and only if so is the given category; see Proposition~\ref{prop:Free}.

Let $\mathcal{C}$ be a finite  category. Recall from \cite[Definition~2.3]{L2011} that a morphism $\alpha\colon x\rightarrow y$ in $\mathcal{C}$
is \emph{unfactorizable}, if it is a non-isomorphism and
whenever it has a factorization $x
\overset{\beta}{\rightarrow} z \overset{\gamma}{\rightarrow} y$,
then either $\beta$ or $\gamma$ is an isomorphism.
We observe that if $\alpha \colon x \rightarrow y$ is unfactorizable, then so is $h\circ\alpha\circ g$ for any isomorphism $h$ and $g$.

As the notion of an unfactorizable morphism is categorical, it is preserved by any categorical automorphism. Then the following observation is clear.

\begin{lem}\label{lem:unf-iso}
Let $G$ be a finite group and  $\mathcal{C}$ be a finite $G$-category. Then for any morphism $\alpha$ in $\mathcal{C}$ and $g\in G$, $\alpha$ is unfactorizable if and only if  so is $g(\alpha)$. \hfill $\square$
\end{lem}

We say that a morphism $\alpha$ in a finite
category $\mathcal{C}$ satisfies the Unique Factorization Property (UFP), if it is either an isomorphism or, whenever it has two  factorizations into unfactorizable morphisms:
\[ x=x_0\overset{\alpha_1}{\rightarrow}
x_1\overset{\alpha_2}{\rightarrow} \cdots
\overset{\alpha_m}{\rightarrow} x_m=y\] and
\[x=y_0\overset{\beta_1}{\rightarrow}
y_1\overset{\beta_2}{\rightarrow} \cdots
\overset{\beta_n}{\rightarrow} y_n=y,\]
then $m=n$, and there are isomorphisms $\gamma_i\colon x_i \rightarrow y_i$ in $\mathcal{C}$ for $1\leq i\leq m-1$, such that the following diagram commutes.
\begin{equation*}
\xymatrix@C=1.0cm{
	x=x_0\ar[r]^{\alpha_1} \ar@{=}[d] & x_1 \ar[r]^{\alpha_2} \ar[d]^{\gamma_1} & x_2 \ar[r]^{\alpha_3} \ar[d]^{\gamma_2} &  \cdots \ar[r]^{\alpha_{m-1}}  & x_{m-1} \ar[r]^{\alpha_m} \ar[d]^{\gamma_{m-1}}& x_m=y \ar@{=}[d] \\
	x=y_0\ar[r]^{\beta_1} & y_1 \ar[r]^{\beta_2} & y_2 \ar[r]^{\beta_3} & \cdots \ar[r]^{\beta_{m-1}}  & y_{m-1} \ar[r]^{\beta_m}& y_m=y
}
\end{equation*}

We mention that, in general,  a non-isomorphism in a finite category $\mathcal{C}$ might not have a factorization into unfactorizable morphisms. However, if $\mathcal{C}$ is EI,  any non-isomorphism in $\mathcal{C}$ has a factorization
into unfactorizable morphisms; see \cite[Proposition~2.6]{L2011}.

\begin{lem}\label{lem:unf}
Let $G$ be a finite group and  $\mathcal{C}$ be a finite $G$-category. Then a morphism $(\alpha,g)$ in $\mathcal{C}\rtimes G$ is unfactorizable if and only if $\alpha$ is unfactorizable in $\mathcal{C}$.
\end{lem}

\begin{proof}
By Lemma~\ref{lem:facts}(1), we observe that $(\alpha,g)$ is a non-isomorphism if and only if so is $\alpha$.

For the ``if" part, we assume that $\alpha$ is unfactorizable. Suppose we have a factorization $$(\alpha,g)=(\beta,h)\circ (\gamma,k)=(\beta\circ h(\gamma),hk)$$
 in $\mathcal{C}\rtimes G$. The factorization $\alpha=\beta\circ h(\gamma)$ in $\mathcal{C}$ implies that
either $\beta$ or $h(\gamma)$ is an isomorphism. As $h\in G$ induces a categorical automorphism on $\mathcal{C}$, we infer that $h(\gamma)$ is an isomorphism if and only if so is $\gamma$. In view of Lemma~\ref{lem:facts}(1), we infer that either $(\beta, h)$ or $(\gamma, k)$ is an isomorphism in $\mathcal{C}\rtimes G$, proving that $(\alpha, g)$ is unfactorizable.

For the ``only if" part, we assume that $(\alpha,g)$ is unfactorizable. Assume on the contrary that $\alpha=\beta\circ \gamma$ with both $\beta$ and $\gamma$ non-isomorphisms in $\mathcal{C}$. Then we have $$(\alpha,g)=(\beta,1_G)\circ (\gamma,g).$$
By Lemma~\ref{lem:facts}(1), we have that both $(\beta,1_G)$ and $(\gamma,g)$ are non-isomorphisms in  $\mathcal{C}\rtimes G$. This contradicts to the unfactorizability of $(\alpha, g)$.
\end{proof}

The following result characterizes  the UFP of morphisms in a skew group category.

\begin{prop}\label{prop:UFP}
	Let $G$ be a finite group and  $\mathcal{C}$ be a finite $G$-category. Then a morphism $(\alpha,g)$ in $\mathcal{C}\rtimes G$ satisfies the UFP  if and only if $\alpha$ satisfies the UFP in $\mathcal{C}$.
\end{prop}

\begin{proof}
 By Lemma~\ref{lem:facts}(1), the morphism $(\alpha, g)$ is an isomorphism if and only if so is $\alpha$. In the following proof, we will assume that both $(\alpha, g)$ and $\alpha\colon g(x)\rightarrow y$ are non-isomorphisms.

For the ``if" part, we assume that $\alpha\colon g(x)\rightarrow y$ satisfies the UFP in $\mathcal{C}$.
Suppose that $(\alpha,g)\colon x\rightarrow y$ has two factroizations into
unfactorizable morphisms in $\mathcal{C}\rtimes G$:
\[ x=x_0\overset{(\alpha_1,g_1)}{\longrightarrow}
x_1\overset{(\alpha_2,g_2)}{\longrightarrow} \cdots
\overset{(\alpha_n,g_n)}{\longrightarrow} x_n=y\] and
\[x=y_0\overset{(\beta_1,h_1)}{\longrightarrow}
y_1\overset{(\beta_2,h_2)}{\longrightarrow} \cdots
\overset{(\beta_m,h_m)}{\longrightarrow} y_m=y.\]
The factorizations imply $g_n\cdots g_1=g=h_m\cdots h_1$ in $G$. Moreover, the morphism $\alpha\colon g(x)\rightarrow y$ has two factorizations in $\mathcal{C}$:
\[g(x)\overset{g_n\cdots g_2(\alpha_1)}{\longrightarrow}
g_n\cdots g_2(x_1)\overset{g_n\cdots g_3(\alpha_2)}{\longrightarrow} \cdots
\overset{g_n(\alpha_{n-1})}{\longrightarrow} g_n(x_{n-1})\overset{\alpha_n}{\longrightarrow} x_n=y\] and
\[g(x)\overset{h_m\cdots h_2(\beta_1)}{\longrightarrow}
h_m\cdots h_2(y_1)\overset{h_m\cdots h_3(\beta_2)}{\longrightarrow} \cdots
\overset{h_m(\beta_{m-1})}{\longrightarrow} h_m(y_{m-1})\overset{\beta_m}{\longrightarrow} y_m=y.\]
Here, in the first factorization we use $g(x)=g_n\cdots g_2g_1(x_0)$, and in the second one we use $g(x)=h_m\cdots h_2h_1(x_0)$. By Lemmas~\ref{lem:unf-iso} and \ref{lem:unf}, all the morphisms appearing in the above two factorizations are unfactorizable in $\mathcal{C}$.

Since $\alpha$ satisfies the UFP, we infer that $n=m$, and that there are isomorphisms $\theta_i: g_n\cdots g_{i+1}(x_i)\rightarrow h_n\cdots h_{i+1}(y_{i})$, $1\leq i\leq n-1$, such that the following diagram in $\mathcal{C}$ commutes.
\begin{equation*}
\xymatrix@C=0.9cm{
	g(x)\ar[rr]^-{g_n\cdots g_2(\alpha_1)} \ar@{=}[d] && g_n\cdots g_2(x_1) \ar[rr]^-{g_n\cdots g_3(\alpha_2)} \ar[d]_{\theta_1}  &&  \cdots \ar[r]^-{g_n(\alpha_{n-1})}  & g_n(x_{n-1}) \ar[r]^{\alpha_n} \ar[d]^{\theta_{n-1}}& x_n=y \ar@{=}[d] \\
	g(x)\ar[rr]^-{h_n\cdots h_2(\beta_1)} && h_n\cdots h_2(y_1) \ar[rr]^-{h_n\cdots h_3(\beta_2)} && \cdots \ar[r]^-{h_n(\beta_{n-1})}  & h_n(y_{n-1}) \ar[r]^{\beta_n}& y_n=y
}
\end{equation*}
Set $\theta_0={\rm Id}_{g(x)}$ and $\theta_n={\rm Id}_y$. Then the above commutativity implies that
 \begin{align}\label{equ:comm-dia}
 \theta_{i}\circ g_n\cdots g_{i+1}(\alpha_i)=h_n\cdots h_{i+1}(\beta_i)\circ \theta_{i-1}
   \end{align}
   for each $1\leq i\leq n$.  The identity for the case $i=n$ means $\alpha_n=\beta_n\circ \theta_{n-1}$.

For each $1\leq i\leq n-1$, we set $a_{i}=h_{i+1}^{-1}\cdots h_n^{-1}g_n\cdots g_{i+1}$. In addition, we set $a_0=1_G=a_n$. Then we have
\begin{align}\label{equ:a}
a_ig_i=h_ia_{i-1}
\end{align}
for each $1\leq i\leq n$. Here, to see $a_1g_1=h_1$, we use the fact that $g_n\cdots g_1=h_n\cdots h_1$.

For each $1\leq i\leq n$, we set $\eta_{i}=h_{i+1}^{-1}\cdots h_n^{-1}(\theta_{i})$. We observe $\eta_0={\rm Id}_x$ and $\eta_n={\rm Id}_y$.  Then each $\eta_i\colon a_i(x_i)\rightarrow y_i$ is an isomorphism in $\mathcal{C}$. Consequently, by Lemma~\ref{lem:facts}(1) we have an isomorphism $(\eta_i, a_i)\colon x_i\rightarrow y_i$ in $\mathcal{C}\rtimes G$.

Applying $h_{i+1}^{-1}\cdots h_n^{-1}$ to (\ref{equ:comm-dia}), we have
\begin{align}\label{equ:aa}
\eta_i\circ  a_i(\alpha_i)=\beta_i\circ h_i(\eta_{i-1})
\end{align}
for each $1\leq i\leq n$. By (\ref{equ:a}) and (\ref{equ:aa}), we  have the following commutative diagram in $\mathcal{C}\rtimes G$.
\begin{equation*}
\xymatrix@C=1.0cm{
	x=x_0\ar[r]^-{(\alpha_1,g_1)} \ar@{=}[d] & x_1 \ar[r]^-{(\alpha_2,g_2)} \ar[d]_{(\eta_1,a_1)} & x_2 \ar[r]^{} \ar[d]^{(\eta_2,a_2)} &  \cdots \ar[r]^{(\alpha_{n-1},g_{n-1})}  & x_{n-1} \ar[r]^-{(\alpha_n,g_n)} \ar[d]^{(\eta_{n-1},a_{n-1})}& x_n=y \ar@{=}[d] \\
	x=y_0\ar[r]^-{(\beta_1,h_1)} & y_1 \ar[r]^-{(\beta_2,h_2)} & y_2 \ar[r]^{} & \cdots \ar[r]^{(\beta_{n-1},h_{n-1})}  & y_{n-1} \ar[r]^-{(\beta_n,h_n)}& y_n=y
}
\end{equation*}
This implies that the morphism $(\alpha,g)$ satisfies the UFP.

The proof of the ``only if"  is similar and actually easier. We assume that $(\alpha,g)\colon x\rightarrow y$ satisfies the UFP in $\mathcal{C}\rtimes G$. Suppose that $\alpha\colon g(x)\rightarrow y$ has two factorizations into unfactorizable morphisms in $\mathcal{C}$:
\[ g(x)=x_0\overset{\alpha_1}{\longrightarrow}
x_1\overset{\alpha_2}{\longrightarrow} \cdots
\overset{\alpha_n}{\longrightarrow} x_n=y\] and
\[g(x)=y_0\overset{\beta_1}{\longrightarrow}
y_1\overset{\beta_2}{\longrightarrow} \cdots
\overset{\beta_m}{\longrightarrow} y_m=y.\]
Then the morphism $(\alpha,g)\colon x\rightarrow y$ in $\mathcal{C}\rtimes G$ has two factorizations:
\[ x=g^{-1}(x_0)\overset{(\alpha_1,g)}{\longrightarrow}
x_1\overset{(\alpha_2,1_G)}{\longrightarrow} \cdots
\overset{(\alpha_n,1_G)}{\longrightarrow} x_n=y\] and
\[x=g^{-1}(y_0)\overset{(\beta_1,g)}{\longrightarrow}
y_1\overset{(\beta_2,1_G)}{\longrightarrow} \cdots
\overset{(\beta_m,1_G)}{\longrightarrow} y_m=y.\]
By Lemma~\ref{lem:unf}, all the morphisms $(\alpha_1,g)$, $(\beta_1,g)$, $(\alpha_i,1_G)$ and $(\beta_j, 1_G)$ are unfactorizable, for $2\leq i\leq n$ and $2\leq j\leq m$.

Since the morphism $(\alpha,g)$ satisfies the UFP, then $m=n$ and there are isomorphisms $(\gamma_i, g_i)\colon x_i\rightarrow y_i$, $1\leq i\leq n-1$, such that the following diagram in $\mathcal{C}\rtimes G$ commutes.
\begin{equation*}
\xymatrix@C=1.0cm{
	x=g^{-1}(x_0)\ar[r]^-{(\alpha_1,g)} \ar@{=}[d] & x_1 \ar[r]^{(\alpha_2,1_G)} \ar[d]_{(\gamma_1,g_1)} & x_2 \ar[r]^{} \ar[d]^{(\gamma_2,g_2)} &  \cdots \ar[r]^-{(\alpha_{n-1}, 1_G)}  & x_{n-1} \ar[r]^-{(\alpha_n,1_G)} \ar[d]^{(\gamma_{n-1},g_{n-1})}& x_n=y \ar@{=}[d] \\
	x=g^{-1}(y_0)\ar[r]^-{(\beta_1,g)} & y_1 \ar[r]^{(\beta_2,1_G)} & y_2 \ar[r]^{} & \cdots \ar[r]^-{(\beta_{n-1}, 1_G)}  & y_{n-1} \ar[r]^-{(\beta_n,1_G)}& y_n=y
}
\end{equation*}
The commutativity implies  $g_1=g_2=\cdots=g_{n-1}=1_G$. By Lemma~\ref{lem:facts}(1), each $\gamma_i\colon x_i\rightarrow y_i$ is an isomorphism in $\mathcal{C}$.  Consequently, the isomorphisms $\gamma_i$ make the following diagram in $\mathcal{C}$ commute.
\begin{equation*}
\xymatrix@C=0.8cm{
	g(x)=x_0\ar[r]^-{\alpha_1} \ar@{=}[d] & x_1 \ar[r]^{\alpha_2} \ar[d]_{\gamma_1} & x_2 \ar[r]^{} \ar[d]^{\gamma_2} &  \cdots \ar[r]^-{\alpha_{n-1}}  & x_{n-1} \ar[r]^-{\alpha_n} \ar[d]^{\gamma_{n-1}}& x_n=y \ar@{=}[d] \\
	g(x)=y_0\ar[r]^-{\beta_1} & y_1 \ar[r]^{\beta_2} & y_2 \ar[r]^{} & \cdots \ar[r]^-{\beta_{n-1}}  & y_{n-1} \ar[r]^-{\beta_n}& y_n=y
}
\end{equation*}
This proves that  $\alpha$ satisfies the UFP, as required.
\end{proof}

Recall that  a finite EI category $\mathcal{C}$ is \emph{free}  provided that  each morphism satisfies the UFP; compare~\cite[Definition~2.7 and Proposition~2.8]{L2011}. For an alternative characterization of a free EI category, we refer to \cite[Proposition~4.5]{Wang}.

The following result follows immediately  from Propositions~\ref{prop:skew-EI} and \ref{prop:UFP}.

\begin{prop}\label{prop:Free}
Let $\mathcal{C}$ be a finite  $G$-category. Then $\mathcal{C}$ is a free EI category  if and only if so is $\mathcal{C}\rtimes G$. \hfill $\square$
\end{prop}

\begin{rem}
Let us sketch a shorter proof of Proposition~\ref{prop:Free} using category algebras. By Proposition ~\ref{prop:skew-EI}, we may assume that both $\mathcal{C}$ and $\mathcal{C}\rtimes G$ are EI categories.

Take an arbitrary field $\mathbb{K}$ of characteristic zero. By Proposition~\ref{prop:skew}, we identify the category algebra  $\mathbb{K}(\mathcal{C}\rtimes G)$ with the skew group algebra $\mathbb{K}\mathcal{C}\# G$. It is well known that $\mathbb{K}\mathcal{C}$ is hereditary if and only if so is $\mathbb{K}\mathcal{C}\#G$; see \cite[Theorems~1.3(c) and 1.4]{RR}. Then Proposition~\ref{prop:Free} follows immediately from the following result due to  \cite[Theorem~5.3]{L2011}: the EI category $\mathcal{C}$ (\emph{resp. }$\mathcal{C}\rtimes G$) is free if and only if the corresponding category algebra $\mathbb{K}\mathcal{C}$ (\emph{resp.} $\mathbb{K}(\mathcal{C}\rtimes G)$) is hereditary.
\end{rem}

\section{Finite EI quivers and $G$-actions}

In this section, we recall basic facts on finite EI quivers. We prove a universal property of the free EI category associated to a finite EI quiver; see Proposition~\ref{prop:universal}. We study finite group actions on finite EI quivers.

\subsection{Categories associated to finite EI quivers} Let $Q=(Q_0, Q_1; s,t)$ be a finite quiver, where $Q_0$ and $Q_1$ are the finite sets of vertices and arrows, respectively. The maps $s, t\colon Q_1\rightarrow Q_0$ assign to each arrow $\alpha$ its starting vertex $s(\alpha)$ and terminating vertex $t(\alpha)$, respectively.

A path $p=\alpha_n\cdots \alpha_2\alpha_1$ of length $n$ in $Q$ consists of arrows $\alpha_i$ satisfying $t(\alpha_i)=s(\alpha_{i+1})$ for each $1\leq i\leq n-1$. Here, we write concatenation from right to left. We set $s(p)=s(\alpha_1)$ and $t(p)=t(\alpha_n)$. An arrow is identified with a path of length one. To each vertex $i\in Q_0$, we associate a trivial path $e_i$ of length zero, satisfying $s(e_i)=i=t(e_i)$.

 A finite  quiver $Q$ is said to be \emph{acyclic}, provided that there is no oriented cycle in $Q$, that is, there is no nontrivial path with the same starting and terminating vertex. This is equivalent to the condition that there are only finitely many paths in $Q$.

Let $H$ be a finite group, and let $X$ be a right $H$-set, that is, $H$ acts on $X$ on the right. Let $Y$ be a left $H$-set.  The \emph{biset product} $X\times_H Y$ is defined to be the set
$$X\times Y/\sim$$
of equivalence classes with respect to the equivalence relation $\sim$ given by $(x.h, y)\sim(x, h.y)$ for $x\in X, h\in H$ and $y\in Y$. By abuse of notation, the elements in $X\times Y/\sim$ are still denoted by $(x,y)$ for $x\in X$ and $y\in Y$.

 Let $G$ and $K$ be finite groups. By a $(G, H)$-biset $X$, we mean a  set $X$ which is a left $G$-set and a right $H$-set satisfying $(g.x).h=g.(x.h)$ for any $g\in G$, $x\in X$ and $h\in H$. Here, we use the dot to denote the group actions. Let $Y$ be a $(H, K)$-biset.  Then the biset product $X\times_H Y$ is naturally a $(G, K)$-biset.

\begin{exm}\label{exm:hom-biset}
Let $\mathcal{C}$ be a finite EI category. For any two objects $x$ and $y$, the ${\rm Hom}$-set ${\rm Hom}_\mathcal{C}(x, y)$ is naturally an $({\rm Aut}_\mathcal{C}(y), {\rm Aut}_\mathcal{C}(x))$-biset, where the actions are given by the composition of morphisms in $\mathcal{C}$.

Denote by ${\rm Hom}^0_\mathcal{C}(x, y)$ the subset of ${\rm Hom}_\mathcal{C}(x, y)$ consisting of unfactorizable morphisms. As unfactorizable morphisms are closed under composition with isomorphisms, ${\rm Hom}^0_\mathcal{C}(x, y)$ is an $({\rm Aut}_\mathcal{C}(y), {\rm Aut}_\mathcal{C}(x))$-sub-biset of ${\rm Hom}_\mathcal{C}(x, y)$.
\end{exm}

Recall from \cite[Definition~2.1]{L2011} that  a \emph{finite EI quiver} $(Q, U)$ consists of a finite acyclic quiver $Q$ and an assignment $U=(U(i), U(\alpha))_{i\in Q_0, \alpha\in Q_1}$. In more details, for each vertex $i\in Q_0$, $U(i)$ is a finite group,  and for each arrow $\alpha\in Q_1$, $U(\alpha)$ is a finite $(U(t\alpha), U(s\alpha))$-biset. Here, we emphasize that each $U(\alpha)$ is nonempty.

For any path $p=\alpha_n\cdots \alpha_2\alpha_1$ in $Q$, we define
$$U(p)=U({\alpha_n}) \times_{U({t\alpha_{n-1}})} U({\alpha_{n-1}}) \times_{U({t\alpha_{n-2}})} \cdots \times_{U({t\alpha_2})} U({\alpha_2})\times_{U({t\alpha_1})} U({\alpha_1})
.$$
Then $U(p)$ is naturally a $(U(tp), U(sp))$-biset. A typical element in $U(p)$ will be denoted by $(u_n, \cdots, u_2, u_1)$ with each $u_i\in U(\alpha_i)$. For each vertex $i\in Q_0$, we identify $U(e_i)$ with $U(i)$.

For two paths $p, q$ satisfying $s(p)=t(q)$, we have a natural isomorphism of $(U(tp), U(sq))$-bisets
\begin{align}\label{iso:1}
U(p)\times_{U(tq)} U(q)\stackrel{\sim}\longrightarrow U(pq),
\end{align}
sending $((u_m',\cdots,  u_1'), (u_n, \cdots, u_1))$ to $(u_m', \cdots, u_1', u_n, \cdots, u_1)$,
where $pq$ denotes the concatenation of paths.

Each finite EI quiver  $(Q, U)$ gives  rise to a finite EI category $\mathcal{C}(Q, U)$; see \cite[Section~2]{L2011}. The objects of $\mathcal{C}(Q, U)$ coincide with the vertices of $Q$. For two objects $i$ and $j$, we have a disjoint union
$${\rm Hom}_{\mathcal{C}(Q, U)}(i, j)=\bigsqcup_{\{p {\rm{ \; paths \; in\;  }} Q {\rm{\; with \; }} s(p)=i  {\rm{\; and \; }} t(p)=j\}} U(p).$$
The composition of morphisms is induced by the concatenation of paths and the isomorphism (\ref{iso:1}).
Since $Q$ has only finitely many paths, we infer that $\mathcal{C}(Q, U)$ is a finite category.  As $e_i$ is the only path starting and terminating at $i$, we infer that
\begin{align}\label{equ:unf-free1}
{\rm Hom}_{\mathcal{C}(Q, U)}(i, i)=U(e_i)=U(i),
\end{align}
which is a finite group. We conclude that the category  $\mathcal{C}(Q, U)$ is indeed finite EI. We mention the following immediate fact
\begin{align}\label{equ:unf-free}
{\rm Hom}^0_{\mathcal{C}(Q, U)}(i, j)=\bigsqcup_{\{\alpha\in Q_1\mid s(\alpha)=i, \; t(\alpha)=j\}} U(\alpha).
\end{align}

By \cite[Proposition~2.8]{L2011}, the EI category $\mathcal{C}(Q, U)$ is free. Moreover,  a finite EI category is free if and only if it is equivalent to  $\mathcal{C}(Q, U)$ for some finite EI quiver $(Q, U)$.

\subsection{A universal property} The free EI category $\mathcal{C}(Q, U)$ enjoys a certain universal property; compare \cite[Proposition~2.9]{L2011}.

\begin{prop}\label{prop:universal}
Let $\mathcal{D}$ be a finite EI category. Assume that $\phi\colon Q_0\rightarrow {\rm Obj}(\mathcal{D})$ is a map, $\psi_i\colon U(i)\rightarrow {\rm Aut}_{\mathcal{D}}(\phi(i))$ is a group homomorphism for each vertex $i\in Q_0$, and that $\psi_\alpha\colon U(\alpha)\rightarrow {\rm Hom}_\mathcal{D}(\phi(s\alpha), \phi(t\alpha))$ is a map of $(U(t\alpha), U(s\alpha))$-bisets for each arrow $\alpha\in Q_1$.  Then there is a unique functor $\Phi\colon \mathcal{C}(Q, U)\rightarrow \mathcal{D}$ subject to the following constraints:
\begin{enumerate}
\item $\Phi(i)=\phi(i)$ for each $i\in Q_0={\rm Obj}(\mathcal{C}(Q, U))$;
\item $\Phi(x)=\psi_i(x)$ for each $x\in U(i)={\rm Aut}_{\mathcal{C}(Q, U)}(i)$;
\item $\Phi(u)=\psi_\alpha(u)$ for each $u\in U(\alpha)\subseteq {\rm Hom}^0_{\mathcal{C}(Q, U)}(s(\alpha), t(\alpha))$.
\end{enumerate}

Moreover, $\Phi$ is an equivalence of categories if and only if all the following conditions are satisfied:
\begin{enumerate}
\item[(E1)] The finite EI category $\mathcal{D}$ is free;
\item[(E2)] Whenever $\phi(i)$ and $\phi(j)$ are isomorphic in $\mathcal{D}$, we have $i=j$;
\item[(E3)] Every object in $\mathcal{D}$ is isomorphic to $\phi(i)$ for some $i\in Q_0$;
\item[(E4)] Each $\psi_i$ is an isomorphism, and the maps $\psi_\alpha$ induce a bijection, for any $i, j\in Q_0$,
    $$
    \bigsqcup_{\{\alpha\in Q_1\mid s(\alpha)=i, t(\alpha)=j\}} U(\alpha)\longrightarrow {\rm Hom}_\mathcal{D}^0(\phi(i), \phi(j)).
    $$
\end{enumerate}
\end{prop}

Before giving the proof, we leave two comments to clarify the statements. The  $(U(t\alpha), U(s\alpha))$-biset structure on ${\rm Hom}_\mathcal{D}(\phi(s\alpha), \phi(t\alpha))$ is given as follows: for a morphism $f\colon \phi(s\alpha)\rightarrow \phi(t\alpha)$ in $\mathcal{D}$, $x\in U(t\alpha)$ and $x'\in U(s\alpha)$, we have
$$x.f.x'=\psi_{t(\alpha)}(x)\circ f\circ \psi_{s(\alpha)}(x').$$
In the condition (E4), the domain of the bijection is a disjoint union; moreover, it implies that $\psi_\alpha(u)$ is unfactorizable in $\mathcal{D}$ for any $u\in U(\alpha)$.

\begin{proof}
Set $\mathcal{C}=\mathcal{C}(Q, U)$. For any path $p=\alpha_n\cdots \alpha_2\alpha_1$ in $Q$ and an element $(u_n, \cdots, u_2, u_1)\in U(p)$, we define
$$\Phi(u_n, \cdots, u_2, u_1)=\psi_{\alpha_n}(u_n)\circ \cdots \circ \psi_{\alpha_2}(u_2)\circ \psi_{\alpha_1}(u_1).$$
We claim that this is independent of the choice of the representatives; compare \cite[the proof of Proposition~2.9]{L2011}.

Assume that $(u_n, \cdots, u_2, u_1)=(v_n, \cdots, v_2, v_1)$ in $U(p)$. This means that there are elements $x_i\in U(t\alpha_i)$ for each $1\leq i\leq n-1$,  such that the following identities hold:
$$v_n=u_n.x_{n-1}, \; v_{i}=x_i^{-1}.u_i.x_{i-1}, \mbox{ and } v_1=x_1^{-1}.u_1.$$
Since each $\psi_{\alpha_i}$ is a map of bisets, we have
$$\psi_{\alpha_n}(v_n)=\psi_{\alpha_n}(u_n)\circ \psi_{t(\alpha_{n-1})}(x_{n-1}),\;  \psi_{\alpha_i}(v_{i})=\psi_{t(\alpha_i)}(x_i)^{-1}\circ \psi_{\alpha_i}(u_i)\circ \psi_{t(\alpha_{i-1})}(x_{i-1}), $$
and $$\psi_{\alpha_1}(v_1)=\psi_{t(\alpha_1)}(x_1)^{-1}\circ \psi_{\alpha_1}(u_1).$$
Then the following identity follows immediately.
\begin{align*}
&\Phi(v_n, \cdots, v_2,v_1)=\psi_{\alpha_n}(v_n)\circ \cdots \circ \psi_{\alpha_2}(v_2)\circ \psi_{\alpha_1}(v_1)\\
&=\psi_{\alpha_n}(u_n) \psi_{t(\alpha_{n-1})}(x_{n-1})\circ \cdots \circ \psi_{t(\alpha_2)}(x_2)^{-1} \psi_{\alpha_2}(u_2)\psi_{t(\alpha_{1})}(x_{1})\circ \psi_{t(\alpha_1)}(x_1)^{-1} \psi_{\alpha_1}(u_1)\\
&=\psi_{\alpha_n}(u_n)\circ \cdots \circ \psi_{\alpha_2}(u_2)\circ \psi_{\alpha_1}(u_1)=\Phi(u_n, \cdots, u_2, u_1).
\end{align*}

The above claim  yields a well-defined functor $\Phi$. The uniqueness of $\Phi$ is clear, as $(u_n, \cdots, u_2, u_1)$ might be viewed as the composition $u_n\circ \cdots \circ u_2\circ u_1$ in $\mathcal{C}$.

For the ``only if" part of the second statement, we assume that $\Phi$ is an equivalence. Then (E1) is clear, since $\mathcal{C}$ is free. Since  $\mathcal{C}$ is skeletal and $\Phi$ respects isomorphism classes, (E2) follows immediately. The condition (E3) is just the denseness of $\Phi$. For (E4), we observe that the equivalence $\Phi$ necessarily induces isomorphisms
$${\rm Aut}_\mathcal{C}(i)\simeq {\rm Aut}_\mathcal{D}(\phi(i))$$
 of groups and bijections
$${\rm Hom}_\mathcal{C}^0(i, j)\simeq {\rm Hom}_\mathcal{D}^0(\phi(i), \phi(j))$$
between the sets of  unfactorizable morphisms. Then we apply (\ref{equ:unf-free1}) and (\ref{equ:unf-free}).

For the ``if" part, we assume the conditions (E1)-(E4).  By (E3), the functor $\Phi$ is dense. It suffices to prove that for any $i, j\in Q_0$, the following map
$$\Phi_{i,j}\colon {\rm Hom}_\mathcal{C}(i, j)\longrightarrow {\rm Hom}_\mathcal{D}(\phi(i), \phi(j)),\quad f\mapsto \Phi(f)$$
is bijective.

By (E4), each $\psi_i$ is an isomorphism, and then the case $i=j$ follows. We now assume that $i\neq j$. Then by (E2), $\phi(i)$ and $\phi(j)$ are not isomorphic.

Recall from \cite[Proposition~2.6]{L2011} that each morphism in $\mathcal{D}$ has a factorization into unfactorizable morphisms. Since $\Phi$ is dense, any morphism $g\colon \phi(i)\rightarrow \phi(j)$ admits a factorization
$$\phi(i)\stackrel{g_0}\longrightarrow \phi(i_1) \stackrel{g_1}\longrightarrow \phi(i_2) \longrightarrow \cdots \longrightarrow \phi(i_{n-1}) \stackrel{g_{n-1}} \longrightarrow \phi(j)$$
with each $g_k$ unfactorizable. By (E4), each $g_k$ belongs to the image of $\Phi$. It follows that there is a morphism $f\colon i\rightarrow j$ in $\mathcal{C}$ satisfying $\Phi(f)=g$. This proves that $\Phi_{i,j}$ is surjective.

It remains to show that $\Phi_{i,j}$ is injective. Assume that $p=\alpha_n\cdots\alpha_2\alpha_1$ and $q=\beta_m\cdots \beta_2\beta_1$ are two paths from $i$ to $j$, and that $(u_n, \cdots, u_2, u_1)\in U(p)$ and $(v_m, \cdots, v_2, v_1)\in U(q)$ satisfy
$$\Phi(u_n, \cdots, u_2, u_1)=\Phi(v_m, \cdots, v_2, v_1)=g'.$$
We claim that $p=q$ and $(u_n, \cdots, u_2, u_1)=(v_m, \cdots, v_2, v_1)$. Then we are done.

For the claim, we observe that  the morphism $g'$ admits two factorizations:
$$\phi(i)\stackrel{\psi_{\alpha_1}(u_1)}\longrightarrow \phi(i_1) \stackrel{\psi_{\alpha_2}(u_2)}\longrightarrow \phi(i_2) \longrightarrow \cdots \longrightarrow \phi(i_{n-1}) \stackrel{\psi_{\alpha_n}(u_n)}\longrightarrow \phi(j)$$
and
$$\phi(i)\stackrel{\psi_{\beta_1}(v_1)}\longrightarrow \phi(j_1) \stackrel{\psi_{\beta_2}(v_2)}\longrightarrow \phi(j_2) \longrightarrow \cdots \longrightarrow \phi(j_{m-1}) \stackrel{\psi_{\beta_m}(v_m)}\longrightarrow \phi(j).$$
Here, $i_k=t(\alpha_k)$ and $j_k=t(\beta_k)$. By (E4), all the morphisms appearing in the two factorizations are unfactorizable. By (E1), the EI category $\mathcal{D}$ is free, that is, any morphism satisfies the UFP. Consequently, $m=n$ and there are isomorphisms $g'_k\colon \phi(i_k)\rightarrow \phi(j_k)$ making the following diagram commute.
\[
\xymatrix{
\phi(i)\ar@{=}[d]  \ar[r]^-{\psi_{\alpha_1}(u_1)} & \phi(i_1) \ar[d]_-{g'_1} \ar[r]^-{\psi_{\alpha_2}(u_2)} & \phi(i_2) \ar[d]^-{g'_2} \ar[r] & \cdots \ar[r] & \phi(i_{n-1}) \ar[d]^-{g'_{n-1}} \ar[r]^-{\psi_{\alpha_n}(u_n)} & \phi(j) \ar@{=}[d]\\
\phi(i) \ar[r]^-{\psi_{\beta_1}(v_1)} & \phi(j_1) \ar[r]^-{\psi_{\beta_2}(v_2)} & \phi(j_2) \ar[r] & \cdots \ar[r] & \phi(j_{n-1}) \ar[r]^-{\psi_{\beta_n}(v_n)} & \phi(j)
}\]
By (E2), we infer that $i_k=j_k$; moreover, by (E4) we obtain automorphisms $a_k\in {\rm Aut}_\mathcal{C}(i_k)=U(i_k)$ satisfying $\psi_{i_k}(a_k)=g'_k$. The commutativity yields
$$\psi_{\beta_{k+1}}(v_{k+1})\circ \psi_{i_k}(a_k)=\psi_{i_{k+1}}(a_{k+1})\circ \psi_{\alpha_{k+1}}(u_{k+1})$$
for each $0\leq k\leq n-1$. Here, $a_0$ and $a_n$ are the identity elements in $U(i)$ and $U(j)$, respectively. The above identity is equivalent to
$$\psi_{\beta_{k+1}}(v_{k+1}.a_k)=\psi_{\alpha_{k+1}}(a_{k+1}.u_{k+1}).$$
By the bijection in (E4), we infer that $\beta_{k+1}=\alpha_{k+1}$ and that
$$v_{k+1}.a_k=a_{k+1}.u_{k+1}$$
for each $0\leq k\leq n-1$. It follows that $p=q$; moreover, in view of the definition of $U(p)$ via biset products, we infer that $(u_n, \cdots, u_2, u_1)=(v_n, \cdots, v_2, v_1)$, proving the claim.
\end{proof}

\subsection{$G$-actions on finite EI quivers}

Let $(Q,U)$ be a finite EI quiver. An \emph{automorphism} $\sigma=(\sigma^0,\sigma^1)$ of $(Q,U)$ consists of an automorphism $\sigma^0\colon Q\rightarrow Q$ of the acyclic quiver $Q$ and an assignment $\sigma^1=(\sigma^1_i,\sigma^1_\alpha)_{i\in Q_0, \alpha\in Q_1}$ of isomorphisms. More precisely, for each $i\in Q_0$,
$$\sigma^1_i\colon U(i)\overset{\sim}{\longrightarrow} U(\sigma^0(i))$$
 is an isomorphism of groups; for each arrow $\alpha\in Q_1$,  $$\sigma^1_\alpha \colon U(\alpha)\overset{\sim}{\longrightarrow} U(\sigma^0(\alpha))$$
  is an isomorphism of  $(U(t\alpha), U(s\alpha))$-bisets. Here, the $(U(t\alpha), U(s\alpha))$-biset structure on $U(\sigma^0(\alpha))$ is induced by the group isomorphisms $\sigma^1_{t(\alpha)}$ and $\sigma^1_{s(\alpha)}$.

The composition of two automorphisms $\sigma=(\sigma^0,\sigma^1)$ and $\theta=(\theta^0,\theta^1)$ on $(Q,U)$ is given by
$$ \theta\circ \sigma =(\theta^0\circ \sigma^0,\theta^1\star\sigma^1),$$
where the assignment $\theta^1\star\sigma^1$ is given by
$$(\theta^1\star\sigma^1)_i=\theta^1_{\sigma^0(i)}\circ \sigma^1_i \mbox{ and } (\theta^1\star\sigma^1)_\alpha=\theta^1_{\sigma^0(\alpha)} \circ \sigma^1_\alpha.$$
We denote by ${\rm Aut}(Q,U)$ the group of automorphisms of $(Q,U)$, whose multiplication is given by the composition of automorphisms.

We observe that each automorphism $\sigma=(\sigma^0, \sigma^1)$ on $(Q, U)$ induces an automorphism $\tilde{\sigma}$ on $\mathcal{C}(Q, U)$ in the following natural manner: the action of $\tilde{\sigma}$ on objects is given by $\sigma^0$; for $u\in U(i)$, we have $\tilde{\sigma}(u)=\sigma^1_i(u)\in U(\sigma^0(i))$;  for a path $p=\alpha_n\cdots \alpha_2\alpha_1$ and a morphism $(u_n, \cdots, u_2, u_1)\in U(p)$, we have
$$\tilde{\sigma}(u_n, \cdots, u_2, u_1)=(\sigma^1_{\alpha_n}(u_n), \cdots, \sigma^1_{\alpha_2}(u_2), \sigma^1_{\alpha_1}(u_1))\in U(\sigma^0(p)).$$
This actually gives rise to an injective group homomorphism
\begin{align}\label{homo:Aut}
{\rm Aut}(Q, U)\hookrightarrow {\rm Aut}(\mathcal{C}(Q, U)), \quad \sigma\mapsto \tilde{\sigma}.
\end{align}

Let $G$ be a finite group. By a \emph{$G$-action} on a finite EI quiver $(Q,U)$, we mean  a group homomorphism
$$\rho\colon G\longrightarrow {\rm Aut}(Q,U),\quad g\mapsto \rho(g)=(\rho(g)^0,\rho(g)^1).$$
Composing $\rho$ with (\ref{homo:Aut}),  the $G$-action makes $\mathcal{C}(Q, U)$ into a $G$-category.

The following convention for the $G$-action $\rho$ will simplify the notation. For $g\in G$ and $i\in Q_0={\rm Obj}(\mathcal{C}(Q,U))$, we write
\begin{align}\label{equ:conven2}
g(i)=\rho(g)^0(i)\in Q_0.
\end{align}
Similarly, for $\alpha\in Q_1$, we write $g(\alpha)=\rho(g)^0(\alpha)\in Q_1$. For $a\in U(i)={\rm Aut}_{\mathcal{C}(Q,U)}(i)$ with $i\in Q_0$, we write
\begin{align}\label{equ:conven3}
g(a)=\rho(g)^1_i(a)\in U(g(i)).
\end{align}
For $u\in U(\alpha)\subseteq {\rm Hom}^0_{\mathcal{C}(Q, U)}(i, j)$ with an arrow $\alpha\in Q_1$ from $i$ to $j$, we write
\begin{align}\label{equ:conven4}
g(u)=\rho(g)^1_\alpha(u)\in U(g(\alpha)).
\end{align}

%Recall from \cite[Subsection~12.1.1]{Lus} that an automorphism $\delta\colon Q\rightarrow Q$ of a finite acyclic quiver $Q$ is admissible, provided that there is no arrow joining two different vertices in the same $\delta$-orbit, or equivalently, there is no path joining two different vertices in the same $\delta$-orbit.

%An automorphism $\sigma=(\sigma^0,\sigma^1)$ of $(Q,U)$ is said to be admissible, provided that the automorphism $\sigma^0$ of the underlying quiver $Q$ is admissible. We say that the above $G$-action $\rho$  is \emph{admissible},  if  the automorphism $\rho(g)$ on $(Q, U)$ is admissible for each $g\in G$.

For admissible $G$-categories, we refer to  Definition~\ref{defn:adm}.

\begin{lem}\label{lem:two-adm}
Let $G$ be a finite group with a $G$-action $\rho$  on $(Q, U)$ as above. Then  the corresponding $G$-category $\mathcal{C}(Q, U)$ is admissible.
\end{lem}

\begin{proof}
Let $g\in G$ and $i\in {\rm Obj}(\mathcal{C}(Q,U))=Q_0$ such that $g(i)\neq i$. Recall that
 $${\rm Hom}_{\mathcal{C}(Q,U)}(g(i),i)=\bigsqcup_{\{p {\rm{ \; paths \; in\;  }} Q {\rm{\; with \; }} s(p)=g(i)  {\rm{\; and \; }} t(p)=i\}} U(p).$$
Since $Q$ is acyclic, there is no path $p$ satisfying $s(p)=g(i)$ and $t(p)=i$. Therefore, the set ${\rm Hom}_{\mathcal{C}(Q,U)}(g(i),i)$ is actually empty, proving that the $G$-category $\mathcal{C}(Q, U)$ is admissible.
\end{proof}

%For the ``if" part, we assume that $\mathcal{C}(Q, U)$ is an admissible $G$-category. We claim that for each $g\in G$, $\rho(g)^0$ is an admissible automorphism of $Q$. Otherwise,  there is a path $p$ starting at some vertex $i\in Q_0$ and ending at $g^r(i)\neq i$, for some $r\geq 1$. The following inclusion
%$$\emptyset\neq U(p)\subseteq {\rm Hom}_{\mathcal{C}(Q,U)}(i, g^r(i))$$
%implies that the $G$-category $\mathcal{C}(Q, U)$ is not admissible, a contradiction.

\section{The quotient EI quiver}

In this section, we fix a finite group $G$ and a finite EI quiver $(Q, U)$ with a  $G$-action $\rho$. We will construct its `orbifold' quotient  EI quiver $(\overline{Q}, \overline{U})$ explicitly. We prove that the free EI category $\mathcal{C}(\overline{Q}, \overline{U})$ is equivalent to the skew group category $\mathcal{C}(Q, U)\rtimes G$; see Theorem~\ref{thm:equiv-cat}.

For each vertex $i$ and each arrow $\alpha$ in $Q$, their stabilizers are denoted by $G_i$ and $G_\alpha$, respectively. We observe that $G_\alpha\subseteq G_{s(\alpha)}\cap G_{t(\alpha)}$.

\subsection{The construction of $(\overline{Q}, \overline{U})$}\label{subsec:quo}

The finite quiver $\overline{Q}=(\overline{Q}_0, \overline{Q}_1; s, t)$ is just the quotient quiver of $Q$ by $G$. In more details,  $\overline{Q}_0=Q_0/G$ and $\overline{Q}_1=Q_1/G$ are the corresponding sets of $G$-orbits, and the maps $s$ and $t$ are induced by the ones of $Q$. Since $Q$ is acyclic, we infer that the finite quiver $\overline{Q}$ is acyclic. By definition, we have the canonical projections
$$\pi_0\colon Q_0\longrightarrow \overline{Q}_0  \mbox{ and } \pi_1\colon Q_1\longrightarrow \overline{Q}_1.$$
The vertices and arrows in $\overline{Q}$ are written in the bold form. For example, the vertices are usually denoted by $\boldsymbol{i}$ and $\boldsymbol{j}$.

To define the assignment $\overline{U}$, we have to fix three maps
\begin{align}\label{choice1}
\iota_0\colon \overline{Q}_0\longrightarrow Q_0,\quad \iota_1\colon \overline{Q}_1\longrightarrow Q_1,\quad \mbox{and } g_{(-)}\colon \overline{Q}_1\longrightarrow G
\end{align}
satisfying  the following conditions: $\pi_0\circ \iota_0={\rm Id}_{\overline{Q}_0}$, $\pi_1\circ \iota_1={\rm Id}_{\overline{Q}_1}$,
 \begin{align}\label{choice1-cond}
 t(\iota_1(\boldsymbol{\alpha}))=\iota_0(t(\boldsymbol{\alpha})) \mbox{ and } s(\iota_1(\boldsymbol{\alpha}))=g_{\boldsymbol{\alpha}}(\iota_0(s\boldsymbol{\alpha}))
 \end{align}
  for each arrow $\boldsymbol{\alpha}\in \overline{Q}_1$. Here, we use the convention (\ref{equ:conven2}) for $g_{\boldsymbol{\alpha}}(\iota_0(s\boldsymbol{\alpha}))$.

The inclusion $G_{\iota_1(\boldsymbol{\alpha})}\subseteq G_{\iota_0(t\boldsymbol{\alpha})}$ makes $G_{\iota_0(t\boldsymbol{\alpha})}$ a right $G_{\iota_1(\boldsymbol{\alpha})}$-set. The injective group homomorphism
$$G_{\iota_1(\boldsymbol{\alpha})}\subseteq G_{s(\iota_1(\boldsymbol{\alpha}))} \overset{\sim}{\longrightarrow} G_{\iota_0(s \boldsymbol{\alpha})}, \quad k\mapsto g_{\boldsymbol{\alpha}}^{-1}kg_{\boldsymbol{\alpha}}$$
makes $G_{\iota_0(s\boldsymbol{\alpha})}$ a left $G_{\iota_1(\boldsymbol{\alpha})}$-set.
Therefore, we have the biset product
 $$G_{\iota_0(t\boldsymbol{\alpha})} \times_{G_{\iota_1(\boldsymbol{\alpha})}} G_{\iota_0(s \boldsymbol{\alpha})},$$
 which is naturally a $(G_{\iota_0(t\boldsymbol{\alpha})},G_{\iota_0(s\boldsymbol{\alpha})})$-biset. A typical element in the above biset product is written as $(h, g)$ with $h\in G_{\iota_0(t\boldsymbol{\alpha})}$ and $g\in G_{\iota_0(s\boldsymbol{\alpha})}$. By definition, we have
 \begin{align}\label{equ:biset-prod}
 (hk, g)=(h, g_{\boldsymbol{\alpha}}^{-1}kg_{\boldsymbol{\alpha}}g)
 \end{align}
 for each $k\in G_{\iota_1(\boldsymbol{\alpha})}$.

 Finally, we choose a right coset decomposition
 \begin{align}\label{choice2}
 G_{\iota_0(t\boldsymbol{\alpha})} = \bigsqcup_{r=1}^{m_{\boldsymbol{\alpha}}}   h_{\boldsymbol{\alpha},r} G_{\iota_1(\boldsymbol{\alpha})}.
 \end{align}
 Consequently, any element in $G_{\iota_0(t\boldsymbol{\alpha})} \times_{G_{\iota_1(\boldsymbol{\alpha})}} G_{\iota_0(s\boldsymbol{\alpha})}$ is uniquely written as
 $(h_{\boldsymbol{\alpha}, r},k)$ for $1\leq r\leq m_{\boldsymbol{\alpha}}$ and $k\in G_{\iota_0(s\boldsymbol{\alpha})}$.

 The construction of the assignment $\overline{U}$ is as follows. For each vertex $\boldsymbol{i}$ of $\overline{Q}$, we set
 $$\overline{U}(\boldsymbol{i})=U(\iota_0(\boldsymbol{i}))\rtimes G_{\iota_0(\boldsymbol{i})}.$$
 Here, we note that $G_{\iota_0(\boldsymbol{i})}$ acts on $U(\iota_0(\boldsymbol{i}))$ by group automorphisms. Therefore, the semi-direct product is well defined. For each arrow $\boldsymbol{\alpha}\colon \boldsymbol{i}\rightarrow \boldsymbol{j}$ in $\overline{Q}$, we set
 \begin{align*}
 \overline{U}(\boldsymbol{\alpha})&=U(\iota_1(\boldsymbol{\alpha})) \times (G_{\iota_0(t\boldsymbol{\alpha})} \times_{G_{\iota_1(\boldsymbol{\alpha})}} G_{\iota_0(s\boldsymbol{\alpha})})\\
 &=U(\iota_1(\boldsymbol{\alpha})) \times (G_{\iota_0(\boldsymbol{j})} \times_{G_{\iota_1(\boldsymbol{\alpha})}} G_{\iota_0(\boldsymbol{i})}).
 \end{align*}
 A typical element in $\overline{U}(\boldsymbol{\alpha})$ is denoted by $(u, (h_{\boldsymbol{\alpha}, r}, k))$. The right $\overline{U}(\boldsymbol{i})$-action is given by
 $$(u,(h_{\boldsymbol{\alpha}, r},k)).(a, g)=(u.(g_{\boldsymbol{\alpha}}k(a)), (h_{\boldsymbol{\alpha}, r},kg))$$
 for any $(a, g)\in \overline{U}(\boldsymbol{i})=U(\iota_0(\boldsymbol{i}))\rtimes G_{\iota_0(\boldsymbol{i})}$.  Here, using the convention (\ref{equ:conven3}),  we observe that $g_{\boldsymbol{\alpha}}k(a)$ lies in $U(g_{\boldsymbol{\alpha}} \iota_0(\boldsymbol{i}))=U(s\iota_1(\boldsymbol{\alpha}))$, and  that $u.(g_{\boldsymbol{\alpha}}k(a))$ means the right $U(s\iota_1(\boldsymbol{\alpha}))$-action on $U(\iota_1(\boldsymbol{\alpha}))$.

 To describe the left $\overline{U}(\boldsymbol{j})$-action, we take an arbitrary element $(b, h)\in \overline{U}(\boldsymbol{j})=U(\iota_0(\boldsymbol{j}))\rtimes G_{\iota_0(\boldsymbol{j})}$. Assume that
 $$h h_{\boldsymbol{\alpha}, r}=h_{\boldsymbol{\alpha}, p} k'$$
 for some $1\leq p\leq m_{\boldsymbol{\alpha}}$ and $k'\in G_{\iota_1(\boldsymbol{\alpha})}$. The left $\overline{U}(\boldsymbol{j})$-action is given by
 \begin{align*}
 (b, h).(u,(h_{\boldsymbol{\alpha}, r},k))&=(h^{-1}_{\boldsymbol{\alpha}, p}(b).k'(u), (h_{\boldsymbol{\alpha},p}k', k))\\
 &=(h^{-1}_{\boldsymbol{\alpha}, p}(b).k'(u), (h_{\boldsymbol{\alpha},p}, g^{-1}_{\boldsymbol{\alpha}}k'g_{\boldsymbol{\alpha}}k)).
 \end{align*}
 Here, we observe that $h^{-1}_{\boldsymbol{\alpha}, p}(b)$ lies in $U(\iota_0(\boldsymbol{j}))=U(t\iota_1(\boldsymbol{\alpha}))$, and that $k'(u)$ lies in $U(\iota_1(\boldsymbol{\alpha}))$. The convention (\ref{equ:conven4}) is used for $k'(u)$.  Finally, $h^{-1}_{\boldsymbol{\alpha}, p}(b).k'(u)$ denotes the left $U(t\iota_1(\boldsymbol{\alpha}))$-action on $U(\iota_1(\boldsymbol{\alpha}))$.

 The above actions make $\overline{U}(\boldsymbol{\alpha})$ a $(\overline{U}(\boldsymbol{j}), \overline{U}(\boldsymbol{i}))$-biset. In summary, we have defined the finite EI quiver $(\overline{Q}, \overline{U})$.

\subsection{An equivalence of categories}

 The EI quiver $(\overline{Q}, \overline{U})$ might be viewed as a certain `orbifold' quotient of $(Q, U)$. We mention a similar construction in the classic work \cite[Section~3]{Bass} on graphs of groups. The EI quiver $(\overline{Q}, \overline{U})$ depends on the choices in (\ref{choice1}) and (\ref{choice2}).

 The following result justifies the quotient construction.

 \begin{thm}\label{thm:equiv-cat}
 Let $(Q, U)$ be a finite EI quiver with a $G$-action $\rho$, and let $(\overline{Q}, \overline{U})$ be its quotient as above. Then there is an equivalence of categories
 $$\mathcal{C}(\overline{Q}, \overline{U})\simeq \mathcal{C}(Q, U)\rtimes G.$$
 \end{thm}

\begin{proof}
 We write $\mathcal{C}=\mathcal{C}(Q, U)$ in this proof. We will apply Proposition~\ref{prop:universal} to deduce the equivalence. We use the map $\iota_0\colon \overline{Q}_0\rightarrow Q_0={\rm Obj}(\mathcal{C}\rtimes G)$, and the identification
 $$\overline{U}(\boldsymbol{i})=U(\iota_0(\boldsymbol{i}))\rtimes G_{\iota_0(\boldsymbol{i})}={\rm Aut}_{\mathcal{C}\rtimes G}(\iota_0(\boldsymbol{i})),$$
 where the right equality follows by combining Lemma~\ref{lem:two-adm} and  Corollary~\ref{cor:EI}.

 To apply  Proposition~\ref{prop:universal}, it remains to construct for each arrow $\boldsymbol{\alpha}$ in $\overline{Q}$, a map between bisets
 $$\overline{U}(\boldsymbol{\alpha})\longrightarrow {\rm Hom}_{\mathcal{C}\rtimes G}(\iota_0(s\boldsymbol{\alpha}), \iota_0(t\boldsymbol{\alpha})).$$
 We will see in the following construction that these maps between bisets yield the required bijection in (E4).

 By Proposition~\ref{prop:Free}, the category $\mathcal{C}\rtimes G$ is EI free, therefore the condition (E1) is satisfied. For two different vertices $\boldsymbol{i}$ and $\boldsymbol{j}$, the vertices $\iota_0(\boldsymbol{i})$ and $\iota_0(\boldsymbol{j})$ are not in the same $G$-orbit. By Lemma~\ref{lem:facts}(2),  $\iota_0(\boldsymbol{i})$ and $\iota_0(\boldsymbol{j})$ are not isomorphic in $\mathcal{C}\rtimes G$, proving the condition (E2).  For each vertex $i\in Q_0$, the corresponding object $i$ is isomorphic to $\iota_0(\pi_0(i))$ in $\mathcal{C}\rtimes G$, proving  (E3). Once we construct the above maps between bisets, we will infer by Proposition~\ref{prop:universal} the required equivalence of categories.

 To construct the required maps, we take arbitrary vertices $\boldsymbol{i}$ and $\boldsymbol{j}$ in $\overline{Q}$. By Lemma~\ref{lem:unf}, we have
 \begin{align*}
{\rm Hom}^0_{\mathcal{C}\rtimes G}(\iota_0(\boldsymbol{i}),\iota_0(\boldsymbol{j}))=&\{(\theta,g)\mid g\in G, \theta\in {\rm Hom}^0_{\mathcal{C}}(g(\iota_0(\boldsymbol{i})),\iota_0(\boldsymbol{j}))\}\\
= &\bigsqcup_{g\in G}  {\rm Hom}^0_{\mathcal{C}}(g(\iota_0(\boldsymbol{i})),\iota_0(\boldsymbol{j}))\times \{g\}. \\
= &\bigsqcup_{g\in G} \; \; \bigsqcup_{\{\alpha\in Q_1\mid s(\alpha)=g(\iota_0(\boldsymbol{i})), \; t(\alpha)=\iota_0(\boldsymbol{j})\}}  U(\alpha)\times \{g\}.
\end{align*}
Here, for the last equality we use (\ref{equ:unf-free}).

For each arrow $\boldsymbol{\alpha}\colon \boldsymbol{i}\rightarrow \boldsymbol{j}$, we define the following subset of ${\rm Hom}^0_{\mathcal{C}\rtimes G}(\iota_0(\boldsymbol{i}),\iota_0(\boldsymbol{j}))$
\begin{align*}
S(\boldsymbol{\alpha})=\bigsqcup_{\{\alpha\in Q_1\mid \pi_1(\alpha)=\boldsymbol{\alpha}, \; t(\alpha)=\iota_0(\boldsymbol{j})\}} \; \bigsqcup_{\{g\in G\mid s(\alpha)=g(\iota_0(\boldsymbol{i})\}} U(\alpha)\times \{g\}.
\end{align*}
Recall from Example~\ref{exm:hom-biset} that the  $(\overline{U}(\boldsymbol{j}), \overline{U}(\boldsymbol{i}))$-biset structure on ${\rm Hom}^0_{\mathcal{C}\rtimes G}(\iota_0(\boldsymbol{i}),\iota_0(\boldsymbol{j}))$ is induced by composition (\ref{equ:composition}) of morphisms in $\mathcal{C}\rtimes G$. Then we infer that $S(\boldsymbol{\alpha})$ is a $(\overline{U}(\boldsymbol{j}), \overline{U}(\boldsymbol{i}))$-sub-biset. Now, we have the following disjoint union
 \begin{align}\label{equ:dis-uni}
 {\rm Hom}^0_{\mathcal{C}\rtimes G}(\iota_0(\boldsymbol{i}),\iota_0(\boldsymbol{j}))=\bigsqcup_{\{\boldsymbol{\alpha}\in \overline{Q}_1\mid s(\boldsymbol{\alpha})=\boldsymbol{i},\; t(\boldsymbol{\alpha})=\boldsymbol{j}\}} S(\boldsymbol{\alpha}).
 \end{align}

 We will complete the proof by establishing an isomorphism of $(\overline{U}(\boldsymbol{j}), \overline{U}(\boldsymbol{i}))$-bisets
 $$\overline{U}(\boldsymbol{\alpha})\simeq S(\boldsymbol{\alpha})$$
 for any arrow $\boldsymbol{\alpha}\colon \boldsymbol{i}\rightarrow \boldsymbol{j}$. Indeed, in view of (\ref{equ:dis-uni}), the isomorphism yields the required bijection in (E4). Then we are done.

 To analyze $S(\boldsymbol{\alpha})$, we observe that in the index set of the outer disjoint union, the arrows $\alpha$ are of the form $h(\iota_1(\boldsymbol{\alpha}))$ for some $h\in G_{\iota_0(\boldsymbol{j})}$. By the coset decomposition (\ref{choice2}), there is a unique $1\leq r\leq m_{\boldsymbol{\alpha}}$ satisfying
$$\alpha=h_{\boldsymbol{\alpha}, r}(\iota_1(\boldsymbol{\alpha})).$$
For simplicity, we write  $h_r$ for $h_{\boldsymbol{\alpha}, r}$. In the inner disjoint union, we have
$$
g(\iota_0(\boldsymbol{i}))=s(\alpha)=h_r(s \iota_1(\boldsymbol{\alpha}))=h_rg_{\boldsymbol{\alpha}} (\iota_0(\boldsymbol{i})).
$$
Consequently, there is a unique $k\in G_{\iota_0(\boldsymbol{i})}$ satisfying $$g=h_rg_{\boldsymbol{\alpha}}k.$$
 The above analysis implies that
$$S(\boldsymbol{\alpha})=\bigsqcup_{r=1}^{m_{\boldsymbol{\alpha}}} \bigsqcup_{k\in G_{\iota_0(\boldsymbol{i})}} U(h_r(\iota_1(\boldsymbol{\alpha})))\times \{h_rg_{\boldsymbol{\alpha}} k\}.$$

We observe that $U(\iota_1(\boldsymbol{\alpha}))$ is bijective to each $U(h_r(\iota_1(\boldsymbol{\alpha})))$ via $\rho(h_r)^1_{\iota_1(\boldsymbol{\alpha})}$, which sends $u\in U(\iota_1(\boldsymbol{\alpha}))$ to $h_r(u)\in U(h_r(\iota_1(\boldsymbol{\alpha})))$. Moreover, the biset product $$G_{\iota_0(\boldsymbol{j})} \times_{G_{\iota_1(\boldsymbol{\alpha})}} G_{\iota_0(\boldsymbol{i})}$$
 is bijective to
$$\{1,2, \cdots, m_{\boldsymbol{\alpha}}\}\times G_{\iota_0(\boldsymbol{i})},$$
which is further bijective to the following disjoint union
$$\bigsqcup_{r=1}^{m_{\boldsymbol{\alpha}}} \bigsqcup_{k\in G_{\iota_0(\boldsymbol{i})}} \{h_r g_{\boldsymbol{\alpha}}k\}.$$
Using these bijections, we infer that the following map
$$\overline{U}(\boldsymbol{\alpha})= U(\iota_1(\boldsymbol{\alpha})) \times (G_{\iota_0(\boldsymbol{j})} \times_{G_{\iota_1(\boldsymbol{\alpha})}} G_{\iota_0(\boldsymbol{i})})\longrightarrow S(\boldsymbol{\alpha}), \;(u, (h_r, k))\mapsto (h_r(u), h_rg_{\boldsymbol{\alpha}} k)$$
is a bijection. We omit the routine verification that this explicit bijection is indeed a map of $(\overline{U}(\boldsymbol{j}), \overline{U}(\boldsymbol{i}))$-bisets. This is the required isomorphism of bisets.
\end{proof}

The above construction of the quotient EI quiver $(\overline{Q}, \overline{U})$ is rather general. In what follows, we impose conditions which will simplify the construction.

\begin{rem}\label{rem:SAT}
Assume that the $G$-action $\rho$ on $(Q, U)$ satisfies the following triviality conditions:
\begin{enumerate}
\item for each $i\in Q_0$, $g\in G_i$ and $a\in U(i)$, we have $g(a)=a$;
\item for each $\alpha\in Q_1$, $g\in G_\alpha$ and $u\in U(\alpha)$, we have $g(u)=u$.
\end{enumerate}
Then the quotient EI quiver $(\overline{Q}, \overline{U})$ is described as follows: $$\overline{U}(\boldsymbol{i})=U(\iota_0(\boldsymbol{i}))\times G_{\iota_0(\boldsymbol{i})}$$
 is the direct product;  for each arrow $\boldsymbol{\alpha}\colon \boldsymbol{i}\rightarrow \boldsymbol{j}$ in $\overline{Q}$, we have
 \begin{align*}
 \overline{U}(\boldsymbol{\alpha})=U(\iota_1(\boldsymbol{\alpha})) \times (G_{\iota_0(\boldsymbol{j})} \times_{G_{\iota_1(\boldsymbol{\alpha})}} G_{\iota_0(\boldsymbol{i})}).
 \end{align*}
 Its typical element is denoted by $(u, (h, k))$ for $u\in U(\iota_1(\boldsymbol{\alpha}))$, $h\in G_{\iota_0(\boldsymbol{j})} $ and $k\in G_{\iota_0(\boldsymbol{i})}$. The right $\overline{U}(\boldsymbol{i})$-action is given by
 $$(u,(h,k)).(a, g)=(u.g_{\boldsymbol{\alpha}}(a), (h, kg)).$$
The left $\overline{U}(\boldsymbol{j})$-action is given by
 \begin{align*}
 (b, g').(u,(h,k))&=(b.u, (g'h, k)).
 \end{align*}
 Here, $u.g_{\boldsymbol{\alpha}}(a)$ and $b.u$ mean the right $U(s\iota_1(\boldsymbol{\alpha}))$-action and the left $U(t\iota_1(\boldsymbol{\alpha}))$-action on $U(\iota_1(\boldsymbol{\alpha}))$, respectively.
\end{rem}

Let $\Delta=(\Delta_0, \Delta_1; s, t)$ be a finite acyclic quiver. Recall that the \emph{path category} $\mathcal{P}_\Delta$ is defined as follows: ${\rm Obj}(\mathcal{P}_\Delta)=\Delta_0$ and ${\rm Hom}_{\mathcal{P}_\Delta}(i, j)$ consists of all paths from $i$ to $j$; the composition is given by concatenation of paths.

Denote by $(\Delta, U_{\rm tr})$ the  EI quiver with trivial assignment $U_{\rm tr}$, that is, each group $U_{\rm tr}(i)$ is trivial and each biset $U_{\rm tr}(\alpha)$ has only one element. We observe
\begin{align}\label{equ:idenf}
\mathcal{C}(\Delta, U_{\rm tr})=\mathcal{P}_\Delta.
\end{align}

Let $G$ be a finite group which acts on $\Delta$ by quiver automorphisms. Then $G$ acts on the associated EI quiver $(\Delta, U_{\rm tr})$.  Denote by $(\overline{\Delta}, \overline{U}_{\rm tr})$ the quotient EI quiver. Therefore, $\overline{\Delta}$ is the quotient quiver of $\Delta$ by $G$.  Fix the choices (\ref{choice1}). In view of Remark~\ref{rem:SAT}, the assignment $\overline{U}_{\rm tr}$ is described as follows: for each vertex $\boldsymbol{i}$ of $\overline{\Delta}$, we have
$$\overline{U}_{\rm tr}(\boldsymbol{i})=G_{\iota_0(\boldsymbol{i})};$$
for each arrow $\boldsymbol{\alpha}\colon \boldsymbol{i}\rightarrow \boldsymbol{j}$ in $\overline{\Delta}$, we have
$$\overline{U}_{\rm tr}(\boldsymbol{\alpha})=G_{\iota_0(\boldsymbol{i})}\times_{G_{\iota_1(\boldsymbol{\alpha})}} G_{\iota_0(\boldsymbol{j})},$$
whose $(G_{\iota_0(\boldsymbol{i})}, G_{\iota_0(\boldsymbol{j})})$-biset structure is given by the  multiplication of  $G_{\iota_0(\boldsymbol{i})}$ from the left, and of $G_{\iota_0(\boldsymbol{j})})$ from the right.

In view of (\ref{equ:idenf}), we have the following special case of Theorem~\ref{thm:equiv-cat}.

\begin{cor}\label{cor:quiver}
Let $\Delta$ be a finite acyclic quiver with a $G$-action. Keep the notation as above. Then there is an equivalence of categories
$$\hskip 130pt
\mathcal{C}(\overline{\Delta}, \overline{U}_{\rm tr})\simeq \mathcal{P}_\Delta\rtimes G.
 \hskip 130pt \square$$
\end{cor}

\section{Categories and algebras associated to Cartan triples}

In this section, we will first recall the algebras \cite{GLS} and EI categories \cite{CW} associated to Cartan triples. For a finite group action on a finite acyclic quiver, we give sufficient conditions on when the quotient EI quiver is of Cartan type. Consequently, the skew group algebra of the path algebra is Morita equivalent to the algebra studied in \cite{GLS}; see Theorem~\ref{thm:quiver}.

For two nonzero integers $a$ and $b$, we denote by ${\rm gcd}(a,b)$ their greatest common divisor, which is always assumed to be positive.

\subsection{Cartan triples}

Let $n\geq 1$ be a positive integer. An $n\times n$ matrix $C=(c_{ij})$ with integer coefficients is called a \emph{symmetrizable generalized Cartan matrix},  provided that the following conditions are satisfied:
\begin{enumerate}
	\item[(C1)] $c_{ii}=2$ for all $i$;
	\item[(C2)] $c_{ij}\leq 0$ for all $i\neq j$, and $c_{ij}<0$ if and only if $c_{ji}<0$;
	\item[(C3)] There is a diagonal  matrix $D={\rm diag}(c_1,\cdots,c_n)$  with $c_i\in  \mathbb{Z}_{\geq 1}$ for all $i$ such that the product matrix $DC$ is symmetric.
\end{enumerate}

The matrix $D$ appearing in (C3) is called a \emph{symmetrizer} of $C$. For brevity, a symmetrizable generalized Cartan matrix is called a Cartan matrix.

Let $C=(c_{ij})$ be a Cartan matrix. An (acyclic) \emph{orientation} of $C$ is a subset $\Omega\subset \{1,2,\cdots,n\}\times \{1,2,\cdots,n\}$ such that the following conditions are satisfied:
\begin{enumerate}
	\item[(O1)] $\{(i,j),(j,i)\}\cap \Omega\neq \emptyset$ if and only if $c_{ij}<0$;
	\item[(O2)] for each sequence $((i_1,i_2),(i_2,i_3),\cdots,(i_t,i_{t+1}))$ with $t\geq 1$ and $(i_s,i_{s+1})\in \Omega$ for all $1\leq s\leq t$,  we have $i_1\neq i_{t+1}$.
\end{enumerate}
Following \cite{CW}, we will call $(C, D, \Omega)$ a \emph{Cartan triple}, where $C$ is a Cartan matrix, $D$ its symmetrizer and $\Omega$ an orientation of $C$.

In what follows, we recall that, associated to each Cartan triple, there are a finite free EI category $\mathcal{C}(C, D, \Omega)$ and a finite dimensional algebra $H(C, D, \Omega)$.

Let $Q=Q(C,\Omega)$ be the finite quiver with the set of vertices $Q_0=\{1,2, \cdots,n\}$ and with the set of arrows
\[Q_1=\{\alpha^{(g)}_{ij}\colon j\rightarrow i\mid(i,j)\in \Omega, 1\leq g\leq {\rm gcd}(c_{ij},c_{ji})\}\sqcup\{\varepsilon_i\colon i\rightarrow i\mid 1\leq i\leq n\}.\]
Let $Q^{\circ}=Q^{\circ}(C,\Omega)$ be the quiver obtained from $Q$ by deleting all the loops $\varepsilon_i$. By the condition (O2), we infer that the finite quiver  $Q^{\circ}$ is acyclic.

We recall the finite EI quiver $(Q^\circ, X)$.  The assignment $X$ is given as follows: $X(i)=\langle \eta_i\; |\; \eta_i^{c_i}=1\rangle$ is a cyclic group of order $c_i$; for each $(i, j)\in \Omega$, we set $G_{ij}=\langle \eta_{ij}\; |\; \eta_{ij}^{{\rm gcd}(c_i, c_j)}=1\rangle$ to be a cyclic group of order ${\rm gcd}(c_i, c_j)$. There are injective group homomorphisms
\[G_{ij}\hookrightarrow X(i), \;  \eta_{ij}\mapsto \eta_i^{\frac{c_i}{{\rm gcd}(c_i, c_j)}}\]
and
\[G_{ij}\hookrightarrow X(j), \;  \eta_{ij}\mapsto \eta_j^{\frac{c_j}{{\rm gcd}(c_i, c_j)}}.\]
Then we have the $(X(i), X(j))$-biset $X(i)\times_{G_{ij}} X(j)$. We set
$$X(\alpha_{ij}^{(g)})=X(i)\times_{G_{ij}} X(j)$$
 for each $1\leq g\leq {\rm gcd}(c_{ij}, c_{ji})$.

 \begin{defn}{\rm (\cite[Definition~4.1]{CW})} \label{defn:CW}
Associated to a Cartan triple $(C, D, \Omega)$,  the finite EI category $\mathcal{C}(C, D, \Omega)$ is defined to be the free EI category $\mathcal{C}(Q^\circ, X)$ associated to the above EI quiver $(Q^\circ, X)$. We say that such  EI quivers $(Q^\circ, X)$ and  EI categories $\mathcal{C}(C, D, \Omega)$ are of \emph{Cartan type}.\hfill $\square$
\end{defn}

Let $\mathbb{K}$  be a field. The following algebras \cite{GLS} play a fundamental role in categorifying the root lattices for non-symmetric Caran matrices. For more background, we refer to \cite{Gei}.

\begin{defn}{\rm (\cite[Section~1.4]{GLS})}\label{defn:GLS}
Let $(C, D, \Omega)$ be a Cartan triple with  $Q=Q(C,\Omega)$. Consider the following $\mathbb{K}$-algebra
\[H(C,D,\Omega)=\mathbb{K}Q/I,\]
where $\mathbb{K}Q$ is the path algebra of $Q$, and $I$ is the two-sided ideal of $\mathbb{K}Q$ generated by  the following set
$$\; \{\varepsilon_k^{c_k}, \; \varepsilon_i^{\frac{c_i}{{\rm gcd}(c_i,c_j)}} \alpha^{(g)}_{ij}-\alpha^{(g)}_{ij}\varepsilon_j^{\frac{c_j}{{\rm gcd}(c_i,c_j)}}\mid k\in Q_0, \;  (i,j)\in \Omega, \; 1\leq g\leq {\rm gcd}(c_{ij},c_{ji})\}.\; \; \; \square$$
\end{defn}

We will recall from \cite[Subsection~4.2]{CW} the construction of a new Cartan triple $(C', D', \Omega')$ from a given one $(C, D, \Omega)$, which depends on the characteristic of $\mathbb{K}$. Recall that $D={\rm diag}(c_1, \cdots, c_n)$.

\vskip 5pt

{\bf{Construction $(\ddag)$}} for the case ${\rm char}(\mathbb{K})=p>0$.\quad  Assume that $c_i=p^{r_i}d_i$ satisfying $r_i\geq 0$ and ${\rm gcd}(p, d_i)=1$. For each $1\leq i, j\leq n$, we set
\[\Sigma_{ij}^p=\{ (l_i,l_j)\; |\;  0\leq l_i< d_i, 0\leq l_j< d_j, l_ip^{r_i}\equiv l_jp^{r_j} ({\rm mod}\ {\rm gcd}(d_i, d_j))\}.\]

The rows and columns of the Cartan matrix $C'$ and its symmetrizer $D'$ are indexed by the following set
$$M=\bigsqcup_{1\leq i\leq n} \{(i, l_i)\; |\; 0\leq l_i<d_i\}.$$
The diagonal entries of $C'$ are $2$, and the off-diagonal entries are given as follows:
 \[c'_{(i, l_i), (j, l_j)}=\begin{cases}
-{\rm gcd}(c_{ij},c_{ji}) p^{r_j-{\rm min}(r_i, r_j)}, & \text{ if } (l_i, l_j)\in \Sigma_{ij}^p; \\
0, & \text{otherwise.}
\end{cases}\]
 Let $D'$ be a diagonal matrix, whose $(i, l_i)$-th component is given by $p^{r_i}$.  Set
$$\Omega'=\{((i, l_i),(j, l_j))\; |\; (i, j)\in \Omega, (l_i, l_j)\in \Sigma_{ij}^p\},$$
which is an orientation of $C'$.

{\bf{Construction $(\ddag)$}} for the case ${\rm char}(\mathbb{K})=0$.\quad  This is very similar to the above construction. We put $d_i=c_i$ and replace $\Sigma_{ij}^p$ by
\begin{align}\label{equ:sigmaij}
\Sigma_{ij}=\{ (l_i,l_j)\; |\;  0\leq l_i< c_i, 0\leq l_j< c_j, l_i\equiv l_j ({\rm mod}\ {\rm gcd}(c_i, c_j))\}.
\end{align}
The off-diagonal entries of $C'$  is given by
 \[c'_{(i, l_i), (j, l_j)}=\begin{cases}
-{\rm gcd}(c_{ij},c_{ji}) , & \text{ if } (l_i, l_j)\in \Sigma_{ij}; \\
0, & \text{otherwise.}
\end{cases}\]
We observe that $C'$ is  symmetric and that  $D'$  is the identity matrix.

  We say that $\mathbb{K}$ has enough roots of unity for $D$, if for each $1\leq i\leq n$, the polynomial $t^{c_i}-1$ splits in $\mathbb{K}[t]$.

\begin{thm}\label{thm:CW}
Assume that $(C, D, \Omega)$ is a Cartan triple and that $\mathbb{K}$ has enough roots of unity for $D$. Keep the notation in Construction $(\ddag)$. Then there is an isomorphism of algebras
$$\mathbb{K}\mathcal{C}(C, D, \Omega)\simeq H(C', D', \Omega').$$
\end{thm}

\begin{proof}
This result is due to \cite[Theorem~4.3]{CW}. We mention that the assumption here  on $\mathbb{K}$ is slightly weaker than the one therein. Since each polynomial $t^{c_i}-1$ splits, we infer that, in Construction $(\ddag)$ for each  case, the ground field $\mathbb{K}$ has a $(\prod_{i=1}^n d_i)$-th primitive root of unity. Then the proof of \cite[Theorem~4.3]{CW}, in particular, the argument in \cite[Section~5]{CW}, carries through under the weaker assumption here.
\end{proof}

We are interested in the following special case.

\begin{prop}\label{prop:CW}
Assume that ${\rm char}(\mathbb{K})=p>0$ and  that $(C, D, \Omega)$ is a Cartan triple such that each $c_i$ is a $p$-power. Then there is an isomorphism of algebras
$$\mathbb{K}\mathcal{C}(C, D, \Omega)\simeq H(C, D, \Omega),$$
which identifies ${\rm Span}_\mathbb{K}\{{\rm Id}_i, \eta_i, \cdots, \eta_i^{c_i-1}\}$ with ${\rm Span}_\mathbb{K}\{e_i, \varepsilon_{i}, \cdots, \varepsilon_{i}^{c_i-1}\}$.
\end{prop}

Here, ${\rm Span}_\mathbb{K}$ means the subspace spanned by the mentioned elements. Both ${\rm Id}_i$ and $\eta_i$ are viewed as automorphisms of $i$ in $\mathcal{C}(C, D, \Omega)$. Similarly, the trivial path $e_i$ and the loop $\varepsilon_i$ are viewed as elements in $H(C, D, \Omega)$.

\begin{proof}
The assumption on entries of $D$ implies that $(C', D', \Omega')=(C, D, \Omega)$, where we identify $(i, 0)\in M$ with $i$; see \cite[Example~6.7]{CW}. For the same reason, the polynomials $t^{c_i}-1$ splits, that is, $\mathbb{K}$ has enough roots of unity for $D$. Then the isomorphism follows from Theorem~\ref{thm:CW}. By the proof of \cite[Theorem~4.3]{CW}, the isomorphism clearly identifies the above two subspaces.
\end{proof}

\subsection{From quotient to Cartan type}\label{subset:qC}

 We study the situation of Corollary~\ref{cor:quiver}. Let $G$ be a finite group and let $\Delta$ be a finite acyclic quiver with a  $G$-action. We give conditions on when the quotient EI quiver $(\overline{\Delta}, \overline{U}_{\rm tr})$ is of Cartan type.

The following natural conditions will be imposed on the quiver $\Delta$.
\begin{enumerate}
\item[($\dagger 1$)] For each $i$, the stabilizer $G_i=\langle \xi_i\mid \xi_i^{a_i}=1\rangle$ is cyclic with order $a_i$.
\item[($\dagger 2$)] For each arrow $\alpha\colon i\rightarrow j$, we have that $\xi_i^{\frac{a_i}{{\rm gcd}(a_i,a_j)}}=\xi_j^{\frac{a_j}{{\rm gcd}(a_i,a_j)}}$ and both belong to $G_\alpha$.
\item[($\dagger 3$)] For each $g\in G$, we have $\xi_{g(i)}=g\xi_ig^{-1}$.
\end{enumerate}
The condition ($\dagger 3$) means that the choice of the specific  generators $\xi_i$ for $G_i$ is compatible with the $G$-action. It follows from ($\dagger 2$) that the inclusion $G_\alpha\subseteq G_i\cap G_j$ is an equality.

Associated to the $G$-action on $\Delta$, we will define a Cartan triple $(C, D, \Omega)$; compare \cite[Section~14.1]{Lus}. The rows and columns of $C$ and $D$ are indexed by the orbit set $\overline{\Delta}_0=\Delta_0/G$. For each $G$-orbit $\boldsymbol{i}$ of vertices,  the corresponding diagonal entry of $D$ is
$$c_{\boldsymbol{i}}=\frac{|G|}{|\boldsymbol{i}|},$$
where $|\boldsymbol{i}|$ denotes the cardinality of the $G$-orbit. For any vertices $\boldsymbol{i}$ and $\boldsymbol{j}$, we denote by $N_{\boldsymbol{i}, \boldsymbol{j}}$ the number of arrows in $\Delta$ between  the $G$-orbit $\boldsymbol{i}$ and $G$-orbit $\boldsymbol{j}$. The corresponding off-diagonal  entry of $C$ is given by
$$c_{\boldsymbol{i}, \boldsymbol{j}}=\frac{-N_{\boldsymbol{i}, \boldsymbol{j}}}{|\boldsymbol{j}|}.$$
The orientation $\Omega$ is consistent with that of $\overline{\Delta}$, that is, $(\boldsymbol{j}, \boldsymbol{i})\in \Omega$ if and only if there is an arrow from $\boldsymbol{i}$ to $\boldsymbol{j}$ in $\overline{\Delta}$.

By the equality $c_{\boldsymbol{i}}c_{\boldsymbol{i}, \boldsymbol{j}}=c_{\boldsymbol{j}}c_{\boldsymbol{j}, \boldsymbol{i}}$, we infer the following useful identity
\begin{align}\label{equ:cc}
\frac{-c_{\boldsymbol{i}, \boldsymbol{j}}}{{\rm gcd}(c_{\boldsymbol{i}, \boldsymbol{j}}, c_{\boldsymbol{j}, \boldsymbol{i}})}=\frac{c_{\boldsymbol{j}}}{{\rm gcd}(c_{\boldsymbol{i}}, c_{\boldsymbol{j}})}.
\end{align}

\begin{thm}\label{thm:quiver}
Assume that the $G$-action on $\Delta$ satisfies $(\dagger 1)\mbox{-}(\dagger 3)$ and that the associated Cartan triple is $(C, D, \Omega)$. Denote by $(Q^\circ, X)$ the corresponding EI quiver of Cartan type. Then there is an isomorphism of EI quivers
$$(\overline{\Delta}, \overline{U}_{\rm tr}) \simeq (Q^\circ, X).$$
Moreover, we have the following immediate consequences.
\begin{enumerate}
\item  There is an equivalence of categories
$$\mathcal{P}_\Delta\rtimes G\simeq \mathcal{C}(C, D, \Omega).$$
\item Assume that $\mathbb{K}$ has enough roots of unity for $D$ and that the Cartan triple $(C', D', \Omega')$ is given in Construction $(\ddag)$. Then the algebras $\mathbb{K}\Delta\# G$ and $H(C', D', \Omega')$    are Morita equivalent.
\item Assume that ${\rm char}(\mathbb{K})=p>0$ and that $G$ is a $p$-group. Then the algebras $\mathbb{K}\Delta\# G$ and $H(C, D, \Omega)$    are Morita equivalent.
\end{enumerate}
\end{thm}

\begin{proof}
We first prove the isomorphism of EI quivers. Take two vertices $\boldsymbol{i}=Gi$ and $\boldsymbol{j}=Gj$ in $\overline{\Delta}$. We observe  $c_{\boldsymbol{i}}=a_i$ and $c_{\boldsymbol{j}}=a_j$. For each arrow $\alpha$ between $i$ and $j$ in $\Delta$, we have observed that $G_\alpha=G_i\cap G_j$, which is of order ${\rm gcd}(|G_i|, |G_j|)$. Then we have
$$|G\alpha|=\frac{|G|}{{\rm gcd}(c_{\boldsymbol{i}}, c_{\boldsymbol{j}})}.$$
It follows that the number of arrows between   $\boldsymbol{i}$ and $\boldsymbol{j}$ in $\overline{\Delta}$ equals
\begin{align*}
\frac{N_{\boldsymbol{i}, \boldsymbol{j}}}{|G\alpha|} &=\frac{-c_{\boldsymbol{i}, \boldsymbol{j}}\;  |\boldsymbol{j}|\;  {\rm gcd}(c_{\boldsymbol{i}}, c_{\boldsymbol{j}})}{|G|}\\
&=\frac{-c_{\boldsymbol{i}, \boldsymbol{j}}\;    {\rm gcd}(c_{\boldsymbol{i}}, c_{\boldsymbol{j}})}{c_{\boldsymbol{j}}}\\
&= {\rm gcd}(c_{\boldsymbol{i}, \boldsymbol{j}}, c_{\boldsymbol{j}, \boldsymbol{i}}).
\end{align*}
Here, the last equality uses (\ref{equ:cc}). Recall that the vertex set of $Q^\circ$ is bijective to the index set of rows of $C$, namely the vertex set $\overline{\Delta}_0$. By comparing the number of arrows, we identify $\overline{\Delta}$ with $Q^\circ$.

We now compare the assignments $\overline{U}_{\rm tr}$ and $X$. By ($\dagger 1$), we infer that $\overline{U}_{\rm tr}(\boldsymbol{i})=G_{\iota_0(\boldsymbol{i})}$ is cyclic of order $c_{\boldsymbol{i}}=a_i$. For each arrow $\boldsymbol{\alpha}\colon \boldsymbol{i}\rightarrow \boldsymbol{j}$ in $\overline{\Delta}$, we write $\iota_0(\boldsymbol{i})=i$ and $\iota_0(\boldsymbol{j})=j$. By (\ref{choice1-cond}), we have the arrow $\iota_1(\boldsymbol{\alpha})\colon g_{\boldsymbol{\alpha}}(i)\rightarrow j$ in $\Delta$. By ($\dagger 2$) and ($\dagger 3$), we have
$$g_{\boldsymbol{\alpha}}\xi_i^{\frac{a_i}{{\rm gcd}(a_i,a_j)}}g_{\boldsymbol{\alpha}}^{-1}=\xi_j^{\frac{a_j}{{\rm gcd}(a_i,a_j)}}.$$
Recall that $\overline{U}_{\rm tr}(\boldsymbol{\alpha})=G_i\times_{G_{\iota_1(\boldsymbol{\alpha})}}G_j$. In view of (\ref{equ:biset-prod}), the above identity implies that, in $\overline{U}_{\rm tr}(\boldsymbol{\alpha})$, we have
$$(\xi_i^{\frac{a_i}{{\rm gcd}(a_i,a_j)}}, 1_G)=(1_G, \xi_j^{\frac{a_j}{{\rm gcd}(a_i,a_j)}}).$$
This actually implies that the following map is well defined
$$\overline{U}_{\rm tr}(\boldsymbol{\alpha})\longrightarrow X(i)\times_{G_{ij}} X(j)=X(\boldsymbol{\alpha}), \quad (\xi_i^a, \xi_j^b)\mapsto (\eta_i^a, \eta_j^b).$$
The above map is bijective and respects the $(G_i, G_j)$-biset structures. We readily deduce that the assignment $\overline{U}_{\rm tr}$ is isomorphic to $X$, as required.

For (1), we recall that $\mathcal{C}(C, D, \Omega)=\mathcal{C}(Q^\circ, X)$. Then the equivalence of categories follows from the obtained isomorphism of EI quivers and Corollary~\ref{cor:quiver}.

For (2) and (3), we recall that the path algebra $\mathbb{K}\Delta$ is identified with the category algebra $\mathbb{K}\mathcal{P}_\Delta$ of the path category. By Proposition~\ref{prop:skew}, we identify $\mathbb{K}(\mathcal{P}_\Delta\rtimes G)$ with $\mathbb{K}\mathcal{P}_\Delta \# G$. Recall from \cite[Proposition~2.2]{We07} that the category algebras of two equivalent categories are Morita equivalent. Applying (1), we infer that $\mathbb{K}\mathcal{C}(C, D, \Omega)$ and $\mathbb{K}(\mathcal{P}_\Delta\rtimes G)$ are Morita equivalent. In summary, we have obtained that $\mathbb{K}\mathcal{C}(C, D, \Omega)$ and $\mathbb{K}\Delta\# G$ are Morit equivalent.

Now, the required statement in (2) follows from the isomorphism in Theorem~\ref{thm:CW}. For (3), we observe that each $c_{\boldsymbol{i}}$ is a $p$-power, as $G$ is a $p$-group. We apply the isomorphism in Proposition~\ref{prop:CW}.
\end{proof}

Although the conditions $(\dagger 1)\mbox{-}(\dagger 3)$ are technical,  as we will see,  natural examples are ubiquitous. The following construction is inspired by \cite[Section~5.3]{Serre}.

\begin{exm}\label{exm:ubi}
Let $n\geq 1$ and $G$ be a finite group. For each $1\leq i\leq n$, we fix $\xi_i\in G$ and assume that $\xi_i$ is of order $a_i$. The cyclic subgroup generated by $\xi_i$ is denoted by $H_i$. The elements $\xi_i$ may not be distinct. Denote by $G/{H_i}$ the set of left $H_i$-cosets, whose elements are denoted by $gH_i$.

We construct an acyclic quiver $\Delta$ as follows: the set $\Delta_0$ of vertices is a disjoint union $\bigsqcup_{i=1}^n G/{H_i}\times \{i\}$; only if  $i<j$ and $\xi_i^{\frac{a_i}{{\rm gcd}(a_i,a_j)}}=\xi_j^{\frac{a_j}{{\rm gcd}(a_i,a_j)}}$, each coset $g(H_i\cap H_j)$ is viewed as an arrow starting at $(gH_i, i)$ and terminating at $(gH_j, j)$. The natural action of $G$ on left cosets induces a $G$-action on $\Delta$. It is trivial to verify that  $(\dagger 1)\mbox{-}(\dagger 3)$ do hold for the $G$-action.
\end{exm}

The following example is our main concern.

\begin{exm}\label{exm:cyclic}
 Let $G=\langle \xi\mid \xi^a=1\rangle $ be a cyclic group of order $a$. Assume that $G$ acts on $\Delta$ such that $G_\alpha=G_{s(\alpha)}\cap G_{t(\alpha)}$ for each arrow $\alpha$ in $\Delta$. Then the conditions  $(\dagger 1)\mbox{-}(\dagger 3)$  are satisfied.

For each vertex $i$ with $|G_{i}|=a_i$, we take the generator $\xi_i$ to be $\xi^{\frac{a}{a_i}}$. Then it is direct to check that $(\dagger 1)\mbox{-}(\dagger 3)$  hold.
\end{exm}

\begin{rem}\label{rem:Cartan}
We observe that any Cartan triple does arise in the  situation of Example~\ref{exm:cyclic}. More precisely, given any Cartan triple $(C, D, \Omega)$, we will construct an acyclic quiver with a cyclic group action such that its associated Cartan triple is the given one; compare \cite[Section~14.1]{Lus}.

Assume that $c={\rm lcm}(c_1, c_2, \cdots, c_n)$ is the least common multiple of the entries of $D$. Set $d_i=\frac{c}{c_i}$. We construct an acyclic  quiver $\Delta$ as follows. The vertex set  and arrow set are given by
$$\Delta_0=\{(i,l_i)\mid 1\leq i\leq n, 0\leq l_i <d_i\},$$
and
$$\Delta_1=\{\alpha^{(g)}_{(i,l_i),(j,l_j)}:(j,l_j)\rightarrow (i,l_i)\mid (i,j)\in \Omega, \ (l_i,l_j)\in \Sigma_{ij},\  1\leq g\leq {\rm gcd}(c_{ij}, c_{ji}) \},$$
respectively, where $\Sigma_{ij}$ is defined in (\ref{equ:sigmaij}).  We observe the following identity $$|\Sigma_{ij}|=\frac{d_id_j}{{\rm gcd}(d_i, d_j)}=\frac{c}{{\rm gcd}(c_i, c_j)}=\frac{-c_{ij}d_j}{{\rm gcd}(c_{ij}, c_{ji})}.$$

Let $G=\langle \sigma \mid \sigma^c=1\rangle$ be a cyclic group of order $c$. Then $G$ acts on $\Delta$ such that
$$\sigma(i, l_i)=(i, l_i+1) \mbox{ and }\sigma(\alpha^{(g)}_{(i, l_i),(j, l_j)})=\alpha^{(g)}_{(i, l_i+1),(j, l_j+1)}.$$
Here, we identify $(i, d_i)$ with $(i, 0)$. This defines a $G$-action on $\Delta$. It is clear that $G_\alpha=G_{s(\alpha)}\cap G_{t(\alpha)}$ for each arrow $\alpha$ in $\Delta$. Then by  Example~\ref{exm:cyclic}, Theorem~\ref{thm:quiver} applies to this $G$-action. Moreover, the associated Cartan triple coincides with the given one.
\end{rem}

The following counter-example shows that the condition $G_{\alpha}=G_{s(\alpha)}\cap G_{t(\alpha)}$ in Example~\ref{exm:cyclic} is necessary.

\begin{exm}
 Let $\Delta$ be  the following Kronecker quiver.
	$$\xymatrix{
	1\ar@/^0.6pc/[r]^{\alpha}\ar@/_0.6pc/[r]_{\beta}&2 }$$
Let $G=\{1_G, \xi\}$ be a cyclic group of order $2$ which acts on $\Delta$  by interchanging  $\alpha$ and $\beta$. We observe that $G_{\alpha}=G_\beta=\{1_G\}\subsetneq G_1\cap G_2=G$. It follows that $(\dag2)$ is not satisfied.

The quotient quiver $\overline{\Delta}$ is of type $A_2$.
$$\xymatrix{\boldsymbol{1} \ar[r]^-{\boldsymbol{\alpha}} & \boldsymbol{2}
}$$
The assignment $\overline{U}_{\rm tr}$ is described as follows: $\overline{U}_{\rm tr}(\boldsymbol{1})=G=\overline{U}_{\rm tr}(\boldsymbol{2})$, and
$$\overline{U}_{\rm tr}(\boldsymbol{\alpha})=G\times G.$$
The quotient EI quiver $(\overline{\Delta}, \overline{U}_{\rm tr})$ is not of Cartan type.
\end{exm}

\section{Induced modules and folding}

In this final section, we first study induced modules on a skew group algebra. For a finite cyclic group action on a finite acyclic quiver, the main goal of the paper, Theorem~\ref{thm:categorify},  constructs a categorification of the folding projection between the relevant root lattices. In the Dynkin cases, the restriction of the categorification to indecomposable modules corresponds to the folding of positive roots; see Proposition~\ref{prop:Dynkin}.

\subsection{Generalities on induced modules}

Let $A$ be a finite dimensional $\mathbb{K}$-algebra, and let $G$ be a finite group acting on $A$ by algebra automorphisms. Denote by $A\# G$ the skew group algebra. We view $A$ as a subalgebra of $A\# G$ by identifying $a\in A$ with $a\#1_G\in A\#G$.

For a left $A$-module $M$, we define a left $A\#G$-module $M\#G$, the \emph{induced module},  as follows: $M\#G=M\otimes \mathbb{K}G$ as a vector space, and its left $A\#G$-action is given by
$$(a\#g).(m\#h)=(gh)^{-1}(a).m\# gh$$
for any $a\#g\in A\#G$ and $m\#h\in M\#G$. There is an isomorphism of left $A\# G$-modules
\begin{align}\label{iso:indu}
(A\# G)\otimes_A M\stackrel{\sim}\longrightarrow  M\#G, \quad (a\# g)\otimes m\mapsto g^{-1}(a).m\#g.
\end{align}

Similarly, for a right $A$-module $N$, we have a right $A\#G$-module $N\#G=N\otimes \mathbb{K}G$ such that its right $A\#G$-action is given by
$$(n\# h).(a\#g)=n.h(a)\#hg.$$
There is an isomorphism of right $A\#G$-modules
\begin{align*}
N\otimes_A(A\# G)\stackrel{\sim}\longrightarrow N\#G, \quad n\otimes (a\# g) \mapsto n.a\#g.
\end{align*}

For each left $A$-module $M$ and $g\in G$, the \emph{twisted} $A$-module ${^gM}$ is defined as follows: $^gM=M$ as a vector space, where an element $m\in M$ corresponds to ${^gm}\in {^gM}$;  its left $A$-action is given by
$$a.{^g{m}}={^g(g(a).m)}$$
for any $a\in A$. This yields the twisting endofunctor ${^g(-)}$ on $A\mbox{-mod}$.

The following facts are contained in \cite[Proposition~1.8]{RR}.

\begin{lem}\label{lem:G-orbit}
Keep the notation as above. Then the following two statements hold.
\begin{enumerate}
\item For each $h\in G$, the $A\#G$-modules $M\# G$ and $(^hM)\#G$ are isomorphic.
\item Assume that both $M$ and $M'$ are indecomposable $A$-modules. Then the $A\#G$-modules $M\#G$ and $M'\#G$ are isomorphic if and only if $M$ and ${^h(M')}$ are isomorphic for some $h\in G$.
\end{enumerate}
\end{lem}

\begin{proof}
For (1), we mention that the isomorphism
$$M\# G\longrightarrow (^hM)\#G$$
sends $m\#g$ to $(^hm)\#gh$.

In view of (1), it remains to prove the ``only if" part of (2). We have a decomposition
$M\#G=\bigoplus_{g\in G} M\# g^{-1}$ of $A$-modules; moreover, the direct summand $M\#g^{-1}$ is isomorphic to ${^gM}$ by identifying $m\#g^{-1}$ with ${^gm}$. We have an isomorphism of $A$-modules
$$M\#G\simeq \bigoplus_{g\in G} {^gM}.$$
Similarly, we have $M'\#G\simeq \bigoplus_{g\in G} {^g{M'}}$. Then the ``only if" part of (2) follows from the Krull-Schmidt theorem.\end{proof}

For an $A$-module $M$, we have a $G$-graded algebra
\begin{align}\label{equ:G-graded}
\bigoplus_{g\in  G} {\rm Hom}_A(M, {^gM})
\end{align}
whose product is given by
$$ff'={^h(f)}\circ f'$$
 for $f\colon M\rightarrow {^gM}$ and $f'\colon M\rightarrow {^hM}$. Here, we use the fact that $^h{(^gM)}={^{gh}M}$,  and observe that  $ff'\colon M\rightarrow {^{gh}M}$ is well defined.

The following isomorphism is well known; see \cite[Section~3]{RR} or \cite[Proposition~2.4]{Chen17}.

\begin{lem}\label{lem:orbit}
Keep the notation as above. Then there is an isomorphism of algebras
\begin{align*}
\bigoplus_{g\in  G} {\rm Hom}_A(M, {^gM})&\stackrel{\sim}\longrightarrow {\rm End}_{A\#G}(M\#G), \\
(f\colon M\rightarrow {^gM}) &\longmapsto (m\#h\mapsto {^{g^{-1}}f(m)}\#hg^{-1}).
\end{align*}
Here, ${^{g^{-1}}f(m)}$ means the element in $M$ that corresponds to $f(m)\in {^gM}$.
\end{lem}

\begin{proof}
We observe that $M\# g^{-1}$ is naturally identified with ${^gM}$ as a left $A$-module. Then the above isomorphism follows from (\ref{iso:indu}) and the Hom-tensor adjunction.
\end{proof}

Denote by $D={\rm Hom}_{\mathbb{K}}(-, \mathbb{K})$ the duality of vector spaces, and by ${\rm Tr}_A(-)$ the transpose of left or right $A$-modules. Recall that the Auslander-Reiten translations are given by $\tau_A=D{\rm Tr}_A$ and $\tau_A^{-}={\rm Tr}_AD$; see \cite[IV]{ARS}.

The following general facts seem to be well known; compare \cite[Lemma~4.2]{RR}.

\begin{lem}\label{lem:Tr}
Let $M$ and $N$ be  a left $A$-module and a right $A$-module, respectively.
\begin{enumerate}
\item There are isomorphisms of right $A\# G$-modules: $DM\#G \simeq D(M\# G)$ and ${\rm Tr}_A(M)\#G\simeq {\rm Tr}(M\#G)$.
    \item There are isomorphisms of left $A\#G$-modules: $DN\#G\simeq D(N\#G)$ and ${\rm Tr}_A(N)\#G\simeq {\rm Tr}(N\#G)$.
        \item There are isomorphisms of left $A\#G$-modules $\tau_A(M)\#G\simeq \tau(M\#G)$ and $\tau^{-}_A(M)\#G\simeq \tau^{-}(M\# G)$.
\end{enumerate}
Here, ${\rm Tr}$ and $\tau$ denote the transpose and Auslander-Reiten translation of $A\# G$-modules, respectively.
\end{lem}

\begin{proof}
We only prove (1), because the proof of (2) is similar, and that (3) follows immediately from (1) and (2).

The first isomorphism is given as follows
\begin{align*}
DM\#G  &\stackrel{\sim}\longrightarrow D(M\#G)\\
\theta\#g &\longmapsto (m\#h\mapsto \theta(m)\delta_{h, g^{-1}}).
\end{align*}
Here, $\delta$ is the Kronecker symbol. For the second one, we first observe a natural isomorphism
\begin{align*}
c_P\colon {\rm Hom}_A(P, A)\#G  &\stackrel{\sim}\longrightarrow {\rm Hom}_{A\#G}(P\#G, A\#G)\\
                \theta\#g & \longmapsto (p\# h\mapsto h(\theta(p))\#hg)
\end{align*}
 of right $A\#G$-modules. Take a minimal projective presentation $P_1\rightarrow P_0\rightarrow M\rightarrow 0$. Recall that the transpose ${\rm Tr}_A(M)$ is defined by the following exact sequence
 \begin{align}\label{equ:Tr}
 {\rm Hom}_A(P_0, A)\longrightarrow {\rm Hom}_A(P_1, A)\longrightarrow {\rm Tr}_A(M)\longrightarrow 0.
 \end{align}

 As ${\rm rad}(A)\#G\subseteq {\rm rad}(A\#G)$, we infer that
 $$P_1\#G \longrightarrow P_0\#G\longrightarrow M\#G \longrightarrow 0$$
 is a minimal projective presentation of $M\#G$. Then the lower exact row of the following commutative diagram follows from the definition of ${\rm Tr}(M\#G)$. The upper exact row is obtained by applying $-\#G$ to (\ref{equ:Tr}).
 \[
 \xymatrix{
 {\rm Hom}_A(P_0, A)\# G \ar[r] \ar[d]_-{c_{P_0}}  &{\rm Hom}_A(P_1, A)\#G \ar[d]_{c_{P_1}}\ar[r] &{\rm Tr}_A(M)\#G \ar[r] & 0\\
 {\rm Hom}_{A\#G}(P_0\# G, A\#G) \ar[r] & {\rm Hom}_{A\#G}(P_1\#G, A\#G) \ar[r] & {\rm Tr}(M\#G) \ar[r] & 0
 }\]
 Then the required isomorphism follows immediately.
\end{proof}

Recall that a finite dimensional algebra $B$ is local provided that  $B/{{\rm rad}(B)}$ is a division algebra. Following \cite[p.65]{ARS}, we say that $B$  is \emph{elementary} if  $B/{{\rm rad}(B)}$ is isomorphic to a product of $\mathbb{K}$.  We observe that a finite dimensional algebra $B$  is local and elementary  if and only if $B/{{\rm rad}(B)}$ is isomorphic to $\mathbb{K}$.

Recall from \cite[Section~1.4]{NV} that a $G$-graded algebra $\Gamma=\bigoplus_{g\in G}\Gamma_g$ is a \emph{crossed product} if each homogeneous component $\Gamma_g$ contains an invertible element. Such a crossed product $\Gamma$ is often denoted by $B\ast G$ with $B=\Gamma_{(1_G)}$.

\begin{lem}\label{lem:le}
Assume that $\mathbb{K}$ is perfect with  ${\rm char}(\mathbb{K})=p>0$ and that $G$ is a finite $p$-group. Let $B$ be a finite dimensional algebra which is local and elementary. Then any crossed product $B\ast G$, as an ungraded algebra,  is local and elementary.
\end{lem}

\begin{proof}
Take a normal subgroup $N$ of $G$ such that $G/N$ is cyclic of order $p$. Then $B\ast G=\bigoplus_{g\in G} B_g$ is naturally $G/N$-graded
$$B\ast G=\bigoplus_{x\in G/N} (\bigoplus_{g\in x} B_g).$$
Under this new grading, it is also a cross product. In other words, we have
$$B\ast G=(B\ast N)\ast G/N.$$
 By induction,  it suffices to prove the statement for the case where $G$ is cyclic of order $p$.

Assume now that $G$ is cyclic of order $p$. We will prove that $B\ast G$ is local and elementary. We will first deal with a special case.

We claim that any crossed product $\mathbb{K}\ast G$  is always local and elementary. Take a generator $g$ of $G$ and an invertible element $u_g$ in $(\mathbb{K}\ast G)_g$. We have $(u_g)^p=\mu \in \mathbb{K}$ for some nonzero $\mu \in \mathbb{K}$. Since $\mathbb{K}$ is perfect, there is some nonzero $\lambda\in \mathbb{K}$ satisfying $\lambda^p=\mu$.  Now the algebra homomorphism $$\mathbb{K}[t]/(t^p)\longrightarrow \mathbb{K}\ast G,$$
sending $t$ to $\lambda^{-1}u_g-1$, is an isomorphism, proving the claim.

For the general case, we observe that  each homogeneous component $B_h$ of $B\ast G$ is a free $B$-module on each side. We observe that $\bigoplus_{h\in G} {\rm rad}(B_h)$ is a two-sided nilpotent ideal of $B\ast G$. Therefore, we have
$$\bigoplus_{h\in G} {\rm rad}(B_h)\subseteq {\rm rad}(B\ast G).$$
Recall that $\mathbb{K}\simeq B/{{\rm rad}(B)}$.  Combining the following obvious isomorphism
 $$B\ast G/{\bigoplus_{h\in G} {\rm rad}(B_h)}\simeq \mathbb{K}\ast G$$
 and the above claim, we infer that $B\ast G$ is local and elementary.
\end{proof}

\begin{lem}\label{lem:lelocal}
Let ${\rm char}(\mathbb{K})=p>0$ and $G$ be a $p$-group. Then the following statements hold.
\begin{enumerate}
\item Assume that $\mathbb{K}$ is perfect. Then any crossed product $\mathbb{K}\ast G$ is local and elementary.
\item Assume that $G$ is cyclic. Then any crossed product $\mathbb{K}\ast G$ is local.
\end{enumerate}
\end{lem}

\begin{proof}
Since (1) is a special case of Lemma~\ref{lem:le}, we only prove (2). Assume that $|G|=q$ for some $p$-power $q$. Take a generator $g$ of $G$ and an invertible element $u_g$ in $(\mathbb{K}\ast G)_g$. We have $(u_g)^q=\mu \in \mathbb{K}$ for some nonzero $\mu \in \mathbb{K}$. We observe a $\mathbb{K}$-algebra isomorphism
$$\mathbb{K}[t]/(t^q-\mu)\stackrel{\sim}\longrightarrow \mathbb{K}\ast G, \quad t\mapsto u_g.$$
Then the required statement follows from a standard fact: the algebra $\mathbb{K}[t]/(t^q-\mu)$ is always local.
\end{proof}

Let us come back to the situation where a finite group $G$ acts on a finite dimensional algebra $A$.

\begin{prop}\label{prop:indec}
Let  ${\rm char}(\mathbb{K})=p>0$ and  $G$ be a finite  $p$-group. Assume that $M$ is a left $A$-module such that ${\rm End}_A(M)$ is local and elementary. Then the following two statements hold.
\begin{enumerate}
\item Assume that $\mathbb{K}$ is perfect. Then ${\rm End}_{A\#G}(M\#G)$ is  local and  elementary.
\item Assume that $G$ is cyclic. Then ${\rm End}_{A\#G}(M\#G)$ is local.
\end{enumerate}
In both cases, the $A\#G$-module $M\#G$ is indecomposable.
\end{prop}

\begin{proof}
Denote the $G$-graded algebra in (\ref{equ:G-graded}) by $\Gamma$. By Lemma~\ref{lem:orbit}, it suffices to prove that $\Gamma$ is local and elementary under the assumption in (1),  and local under the assumption in  (2), respectively.

Consider the following $G$-graded subspace of $\Gamma$
$$I=\bigoplus_{g\in G} \{f\in {\rm Hom}_A(M, {^gM})\; |\; f \mbox{  is a non-isomorphism}\}.$$
Since all the $A$-modules ${^gM}$ are indecomposable, it follows that $I$ is a $G$-graded two-sided ideal of $\Gamma$. Moreover, it is well known to be nilpotent; for example, see \cite[VI, Corollary~1.3]{ARS}. Consequently, we have $I\subseteq {\rm rad}(\Gamma)$.

Consider the stabilizer $G_M=\{g\in G\; |\; M\simeq {^gM}\}$ of $M$. Recall that
$\mathbb{K}\simeq {{\rm End}_A(M)}/{{\rm rad}({\rm End}_A(M))}$, since ${\rm End}_A(M)$ is local and elementary. We infer that $\Gamma/I$ is isomorphic to a crossed product $\mathbb{K}\ast G_M$.

By the inclusion $I\subseteq {\rm rad}(\Gamma)$, we have
$$\Gamma/{{\rm rad}(\Gamma)}\simeq \mathbb{K}\ast G_M/{{\rm rad}(\mathbb{K}\ast G_M)}.$$
As $G_M$ is a $p$-group, we can apply Lemma~\ref{lem:lelocal} to $\mathbb{K}\ast G_M$. Then the required statements follow immediately.
\end{proof}

\begin{rem}
Assume that $\mathbb{K}$ is algebraically closed in Proposition~\ref{prop:indec}. Then for each indecomposable $A$-module $M$,  ${\rm End}_A(M)$ is local and automatically elementary. It follows that the $A\#G$-module $M\#G$ is indecomposable. In other words, the induction functor
$$-\#G\colon A\mbox{-mod}\longrightarrow A\#G\mbox{-mod}$$
preserves indecomposable modules.
\end{rem}

\subsection{The folding projection and categorification} In this final subsection, we always work in the following setup.
\vskip 5pt

 {\bf{Setup ($\clubsuit$)}}. \; Let $\mathbb{K}$ be a field with ${\rm char}(\mathbb{K})=p>0$, and let $G=\langle \sigma\; |\; \sigma^{p^a}=1_G\rangle$ be a cyclic group of order $p^a$ for some $a\geq 1$. Let $\Delta$ be a finite acyclic quiver with $\Delta_0=\{1, 2, \cdots, n\}$. Assume that $G$ acts on $\Delta$ by quiver automorphisms such that for each arrow $\alpha\in \Delta_1$, we have $G_{\alpha}=G_{s(\alpha)}\cap G_{t(\alpha)}$.

\vskip 5pt

Denote by $\mathbb{Z}\Delta_0=\bigoplus_{i=1}^n \mathbb{Z}\epsilon_i$ the root lattice of $\Delta$. It is endowed with a symmetric bilinear form given by $(\epsilon_i, \epsilon_i)=2$ and
$$(\epsilon_i, \epsilon_j)=-|\{\mbox{arrows between } i \mbox{ and } j \mbox{  in } \Delta\}|$$
for $i\neq j$. Denote by $\Phi^+(\Delta)$ the set of positive roots \cite{Kac}.

Denote by $\Delta_0/G$ the orbit set of vertices. The elements in $\Delta_0/G$ are denoted in the bold form.
The canonical projection $\pi\colon \Delta_0\rightarrow \Delta_0/G$ sends $i$ to $\pi(i)=Gi=\boldsymbol{i}$.

Associated to the $G$-action on $\Delta$, we have defined a Cartan triple $(C, D, \Omega)$ in Subsection~\ref{subset:qC}. The rows and columns of $C$ and $D$ are indexed by $\Delta_0/G$. The entries $c_{\boldsymbol{i}}$ of $D$ are determined by
$$c_{\boldsymbol{i}} =\frac{|G|}{|\boldsymbol{i}|}=p^{a_{\boldsymbol{i}}}$$
for some $0\leq a_{\boldsymbol{i}}\leq a$.  The corresponding root lattice $\mathbb{Z}(\Delta_0/G)=\bigoplus_{{\boldsymbol{i}}\in \Delta_0/G}\mathbb{Z}E_{\boldsymbol{i}}$ is endowed with a symmetric bilinear form given by $(E_{\boldsymbol{i}}, E_{\boldsymbol{i}})=2c_{\boldsymbol{i}}$ and
$$(E_{\boldsymbol{i}}, E_{\boldsymbol{j}})=c_{\boldsymbol{i}}c_{\boldsymbol{i}, \boldsymbol{j}}=-\frac{|G|}{|\boldsymbol{i}|\cdot |\boldsymbol{j}|}\cdot |\{\mbox{arrows between }  G\mbox{-orbits } \boldsymbol{i} \mbox{ and } \boldsymbol{j} \mbox{  in } \Delta\}|$$
for $\boldsymbol{i}\neq \boldsymbol{j}$. Denote by $\Phi^+(C)$ the set of positive roots.

There is a canonical projection between the root lattices
$$\boldsymbol{f}\colon \mathbb{Z}\Delta_0\longrightarrow \mathbb{Z}(\Delta_0/G)$$
given by $\boldsymbol{f}(\epsilon_{i})=E_{\pi(i)}$, and  called the \emph{folding projection}; see \cite[Section~10.3]{Spr}. It does not preserve the bilinear forms. However, it sends positive roots to positive roots. Moreover, by adapting the proof of \cite[Lemma~5.3]{Kac},  \cite[Proposition~15]{Hu04} proves that $\boldsymbol{f}$ restricts to a surjective map
\begin{align*}
\boldsymbol{f}\colon \Phi^+(\Delta) \twoheadrightarrow \Phi^+(C).
\end{align*}
We observe that the folding projection induces an isomorphism between the quotient group of $G$-coinvariants in $\mathbb{Z}\Delta_0$ and  $\mathbb{Z}(\Delta_0/G)$.

We mention that there is a \emph{folding inclusion} from the dual root lattice of $C$ into $\mathbb{Z}\Delta_0$, which identifies the dual root lattice with the subgroup of $G$-invariants in $\mathbb{Z}\Delta_0$; see \cite[Section~2]{Hu04}. Working with species over a finite field, one observes that the extension-of-scalars functor along a suitable field extension yields  a categorification of the folding inclusion; see \cite[the proof of Theorem~24]{Hu04} and \cite[Section~9]{DD}.

Recall that $\mathcal{P}_\Delta$ is the path category of $\Delta$. Then the $G$-action on $\Delta$ induces a $G$-action on $\mathcal{P}_\Delta$. We have the corresponding skew group category $\mathcal{P}_\Delta\rtimes G$.

We identify the path algebra $\mathbb{K}\Delta$ with the category algebra $\mathbb{K}\mathcal{P}_\Delta$. By Proposition~\ref{prop:skew}, we have the following natural isomorphism of algebras
\begin{align}\label{iso:skew}
\varpi\colon \mathbb{K}(\mathcal{P}_\Delta\rtimes G) \stackrel{\sim}\longrightarrow \mathbb{K}\Delta\# G, \quad (q, g)\mapsto q\#g,
\end{align}
for any path $q$ in $\Delta$ and $g\in G$.

Set $\mathcal{C}=\mathcal{C}(C, D, \Omega)$ to be the EI category associated to $(C, D,\Omega)$; see Definition~\ref{defn:CW}. For each $\boldsymbol{i}\in {\rm Obj}(\mathcal{C})=\Delta_0/G$, we have
$${\rm Aut}_\mathcal{C}(\boldsymbol{i}) = \langle \eta_{\boldsymbol{i}}\; |\; \eta_{\boldsymbol{i}}^{c_{\boldsymbol{i}}}={\rm Id}_{\boldsymbol{i}}\rangle, $$
which is a cyclic group of order $c_{{\boldsymbol{i}}}=p^{a_{{\boldsymbol{i}}}}$.

By Theorem~\ref{thm:quiver}, we identify $\mathcal{C}$ with $\mathcal{C}(\overline{\Delta}, \overline{U}_{\rm tr})$. Fix the choices (\ref{choice1}) for the $G$-action on $\mathcal{P}_\Delta$. Then we obtain an equivalence of categories
\begin{align}\label{iso:equiv}
\iota\colon \mathcal{C} \stackrel{\sim}\longrightarrow \mathcal{P}_\Delta\rtimes G
\end{align}
which satisfies $\iota(\boldsymbol{i})=\iota_0(\boldsymbol{i})$. The functor $\iota$ induces the following isomorphism of groups
\begin{align}\label{iso:group}
{\rm Aut}_\mathcal{C}(\boldsymbol{i}) \stackrel{\sim}\longrightarrow G_{\iota_0(\boldsymbol{i})}={\rm Aut}_{\mathcal{P}_\Delta \rtimes G}(\iota_0(\boldsymbol{i})), \quad \eta_{\boldsymbol{i}}\mapsto \sigma^{p^{a-a_{\boldsymbol{i}}}}.
\end{align}

Recall from Definition~\ref{defn:GLS} the algebra $H=H(C, D, \Omega)$. Each $\boldsymbol{i}$ corresponds to an idempotent $e_{\boldsymbol{i}}$ of $H$. Moreover, we have
$$e_{\boldsymbol{i}}He_{\boldsymbol{i}}={\rm Span}_\mathbb{K}\{e_{\boldsymbol{i}}, \varepsilon_{\boldsymbol{i}}, \cdots, \varepsilon_{\boldsymbol{i}}^{c_{\boldsymbol{i}}-1}\}.$$

By Proposition~\ref{prop:CW}, there is an isomorphism of algebras
\begin{align}\label{iso:HKC}
\theta\colon H\stackrel{\sim}\longrightarrow \mathbb{K}\mathcal{C}
\end{align}
which identifies $e_{\boldsymbol{i}}He_{\boldsymbol{i}}$ with $\mathbb{K}{\rm Aut}_\mathcal{C}(\boldsymbol{i})$. Indeed, we have $\theta(e_{\boldsymbol{i}})={\rm Id}_{\boldsymbol{i}}$ and $\theta(\varepsilon_{\boldsymbol{i}})=\eta_{\boldsymbol{i}}-{\rm Id}_{\boldsymbol{i}}$.

We now combine (\ref{iso:skew}), (\ref{iso:equiv}) and (\ref{iso:HKC}) into the following sequence of equivalences. \[\xymatrix{
\mathbb{K}\Delta\#G\mbox{-mod} \ar[r]^-{\varpi^*}& \mathbb{K}(\mathcal{P}_\Delta\rtimes G)\mbox{-mod} \ar@{=}[r]^-{\rm can} & (\mathbb{K}\mbox{-mod})^{\mathcal{P}_\Delta\rtimes G}\ar[d]^-{\iota^*}\\
H\mbox{-mod}&  \ar[l]_-{\theta^*} \mathbb{K}\mathcal{C}\mbox{-mod} & \ar@{=}[l]_-{\rm can} (\mathbb{K}\mbox{-mod})^\mathcal{C}
}\]
Here, the two can's mean the canonical equivalence in (\ref{equ:can-C}), and the upper star functors are given by restriction of scalars. For example, $\iota^*$ sends a functor $X$ on $\mathcal{P}\rtimes G$ to the composite functor $X\circ \iota$. We compose the sequence into an equivalence
$$\Psi\colon \mathbb{K}\Delta\#G\mbox{-mod}\stackrel{\sim}\longrightarrow H\mbox{-mod}.$$

The following terminology is introduced in \cite[Definition~1.1 and Section~11]{GLS}. A left $H$-module $Y$ is \emph{locally free}, provided that each $e_{\boldsymbol{i}}Y$, as an $e_{\boldsymbol{i}}He_{\boldsymbol{i}}$-module, is free. For such a module, its   rank vector is defined as follows
$$\underline{\rm rank}(Y)=\sum_{\boldsymbol{i}\in \Delta_0/G} {\rm rank}_{e_{\boldsymbol{i}}He_{\boldsymbol{i}}}(e_{\boldsymbol{i}}Y) E_{\boldsymbol{i}}\in \mathbb{Z}(\Delta_0/G).$$
A left $H$-module  $Y$ is called \emph{$\tau$-locally free}, provided that for any $k\in \mathbb{Z}$, $\tau_H^{k}(Y)$ is locally free. Slightly different from \cite{GLS}, we do not require $\tau$-locally free $H$-modules to be indecomposable.

\begin{thm}\label{thm:categorify}
Keep the assumptions in Setup ($\clubsuit$).  Let $M$ be a left $\mathbb{K}\Delta$-module. Then $\Psi(M\#G)$ is a $\tau$-locally free $H$-module satisfying
\begin{align}\label{equ:categorify}
\underline{\rm rank}\;\Psi(M\#G)=\boldsymbol{f}(\underline{\rm dim}\; M).
\end{align}

Assume further that ${\rm End}_{\mathbb{K}\Delta}(M)$ is local and elementary. Then ${\rm End}_H(\Psi(M\#G))$ is local. If moreover $\mathbb{K}$ is perfect, then ${\rm End}_H(\Psi(M\#G))$ is local and elementary.
\end{thm}

Denote by $H\mbox{-mod}^{\tau \mbox{-lf}}$ the full subcategory of $H\mbox{-mod}$ consisting of $\tau$-locally free modules. The identity (\ref{equ:categorify}) might be visualized as a commutative diagram.
\[
\xymatrix{
\mathbb{K}\Delta\mbox{-mod} \ar[d]_-{\underline{\rm dim}} \ar[rr]^-{\Psi\circ(-\#G)} && H\mbox{-mod}^{\tau \mbox{-lf}} \ar[d]^-{\underline{\rm rank}}\\
\mathbb{Z}\Delta_0\ar[rr]^-{\boldsymbol{f}} && \mathbb{Z}(\Delta_0/G)
}
\]
The diagram indicates that the composite functor $\Psi\circ(-\#G)$ categorifies the folding projection $\boldsymbol{f}$ between the root lattices.

It is natural to categorify the folding projection $\boldsymbol{f}\colon \Phi^+(\Delta)\rightarrow \Phi^+(C)$ between the positive roots using the same functor between indecomposable modules. However, we have to restrict to the Dynkin cases; see Proposition~\ref{prop:Dynkin}.

\begin{proof}
 \emph{Step 1.}\; We first show that the $H$-module $\Psi(M\#G)$ is locally free and satisfies the required identity for the rank vector.

Recall that the isomorphism $\theta$ identifies $e_{\boldsymbol{i}}He_{\boldsymbol{i}}$ with $\mathbb{K}{\rm Aut}_\mathcal{C}(\boldsymbol{i})$. Therefore, it suffices to claim that for each $\boldsymbol{i}\in {\rm Obj}(\mathcal{C})=\Delta_0/G$,
$$\iota^*\circ {\rm can}\circ \varpi^*(M\# G)(\boldsymbol{i})$$
 is a free module over $\mathbb{K}{\rm Aut}_\mathcal{C}(\boldsymbol{i})$ with rank
 $$\sum_{i\in \boldsymbol{i}} {\rm dim}_\mathbb{K}(e_iM).$$
  Here, we view $\iota^*\circ {\rm can}\circ \varpi^*(M\# G)$ as a functor over $\mathcal{C}$.

For the claim, we observe the following identity.
\begin{align*}
\iota^*\circ {\rm can}\circ \varpi^*(M\# G)(\boldsymbol{i})&={\rm can}\circ \varpi^*(M\#G)(\iota_0(\boldsymbol{i}))\\
&=(e_{\iota_0(\boldsymbol{i})}\#1_G).(M\#G)\\
&=\bigoplus_{g\in G}e_{g(\iota_0(\boldsymbol{i}))}M\#g^{-1}\\
&=\bigoplus_{i\in \boldsymbol{i}} e_iM\#\{g\in G\; |\; g^{-1}(\iota_0(\boldsymbol{i}))=i\}
\end{align*}
Here, for the second equality we recall that the trivial path $e_{\iota_0(\boldsymbol{i})}$ is the identity endomorphism of $\iota_0(\boldsymbol{i})$ in $\mathcal{P}_\Delta$, and for the third one, we use the fact that $g(e_{\iota_0(\boldsymbol{i})})=e_{g(\iota_0(\boldsymbol{i}))}$.

By (\ref{iso:group}), we identify $\mathbb{K}{\rm Aut}_\mathcal{C}(\boldsymbol{i})$ with $\mathbb{K}G_{\iota_0(\boldsymbol{i})}$. As $G$ is abelian, we have $G_{\iota_0(\boldsymbol{i})}=G_i$ for each $i\in \boldsymbol{i}$. Then we observe that the left $\mathbb{K}G_{\iota_0(\boldsymbol{i})}$-action on the above direct summand
$$ e_iM\#\{g\in G\; |\; g^{-1}(\iota_0(\boldsymbol{i}))=i\}$$
is really only on the right side, that is, on the set $\{g\in G\; |\; g^{-1}(\iota_0(\boldsymbol{i}))=i\}$ via the multiplication in $G$. The latter $G_{\iota_0(\boldsymbol{i})}$-action is free and transitive. Therefore, the $G_{\iota_0(\boldsymbol{i})}$-action on the above direct summand is free of rank ${\rm dim}_\mathbb{K}(e_iM)$. This observation implies the claim.

\emph{Step 2.}\; Since $\Psi$ is an equivalence, it commutes with Auslander-Reiten translations. Then we have isomorphisms
$$\tau_H^k\Psi(M\#G)\simeq \Psi \tau^k(M\# G)\simeq \Psi(\tau_{\mathbb{K}\Delta}^k(M)\#G),$$
where the isomorphism on the right side follows from Lemma~\ref{lem:Tr}. Here, the unadorned $\tau$ means the Auslander-Reiten translation of $\mathbb{K}\Delta\# G$-modules. By Step~1, we infer that each $H$-module $\tau_H^k\Psi(M\#G)$  is locally free, that is, $\Psi(M\#G)$ is {$\tau$-locally} free.

The equivalence $\Psi$ induces an isomorphism of algebras
$${\rm End}_H(\Psi(M\#G))\simeq {\rm End}_{\mathbb{K}\Delta\#G}(M\#G).$$
Then the last statement follows from Proposition~\ref{prop:indec}.
\end{proof}

Denote by $\mathbb{K}\Delta\mbox{-ind}$ a complete set of representatives of indecomposable $\mathbb{K}\Delta$-modules. Similarly,  $H\mbox{-ind}^{\tau \mbox{-lf}}$ is a complete set of representatives of indecomposable $\tau$-locally free $H$-modules. As $G$ acts on $\mathbb{K}\Delta\mbox{-ind}$ by twisting endofunctors, we have the orbit set $\mathbb{K}\Delta\mbox{-ind}/G$.

\begin{prop}\label{prop:Dynkin}
Keep the assumptions in Setup ($\clubsuit$). We assume further that $\Delta$ is of Dynkin type. Then the following commutative diagram is well defined
\[
\xymatrix{
\mathbb{K}\Delta\mbox{-{\rm ind}} \ar[d]_-{\underline{\rm dim}} \ar[rr]^-{\Psi\circ(-\#G)} && H\mbox{-{\rm ind}}^{\tau \mbox{-{\rm lf}}} \ar[d]^-{\underline{\rm rank}}\\
\Phi^+(\Delta)\ar[rr]^-{\boldsymbol{f}} && \Phi^+(C),
}
\]
whose vertical arrows are bijections. In particular, the map $\Psi\circ (-\#G)$ induces a bijection
\begin{align*}
\mathbb{K}\Delta\mbox{-{\rm ind}}/G \stackrel{\sim}\longrightarrow H\mbox{-{\rm ind}}^{\tau \mbox{-{\rm lf}}}.
\end{align*}
\end{prop}

\begin{proof}
We observe that $C$ is also of Dynkin type; compare \cite[Proposition~6.5]{CW}. The map $\underline{\rm dim}$ is bijective by the well-known Gabriel's theorem; see \cite[1.2 Satz]{Gab} and \cite[VIII.5]{ARS}. By \cite[Theorem~1.3]{GLS}, the map $\underline{\rm rank}$ is bijective.

 It is well known that  each indecomposable $\mathbb{K}\Delta$-module $M$ satisfies ${\rm End}_{\mathbb{K}\Delta}(M)\simeq \mathbb{K}$; for example, see \cite[VIII, Lemma~6.1]{ARS}. We infer from Theorem~\ref{thm:categorify} that the $H$-module $\Psi(M\#G)$ is  indecomposable. Then the above commutative diagram is well defined. Since $\boldsymbol{f}\colon \Phi^+(\Delta)\rightarrow \Phi^+(C)$ is surjective, we infer that the  map
 $$\Psi\circ (-\#G)\colon \mathbb{K}\Delta\mbox{-{\rm ind}} \longrightarrow H\mbox{-{\rm ind}}^{\tau \mbox{-{\rm lf}}}$$ is  surjective.
 In view of Lemma~\ref{lem:G-orbit}, we have the induced bijection.
\end{proof}

\begin{rem}\label{rem:categorify}
(1) We mention that any non-symmetric  Cartan matrix $C$ of Dynkin type does appear in the situation of Proposition~\ref{prop:Dynkin}; see \cite[14.1.6]{Lus} or \cite[p.81, Table~1]{CW}. The representation theory related to the folding inclusion in the Dynkin cases is studied in \cite{Tani}.

(2) Since  \cite[Theorem~1.3]{GLS} works currently only for Dynkin cases,  we do not know how to extend Proposition~\ref{prop:Dynkin} to non-Dynkin quivers.

In view of \cite{Kac80} and \cite[3.3~Theorem]{DX}, the following open question, analogous to Kac's theorem,  is very natural: does the set of rank vectors of  indecomposable $\tau$-locally free $H$-modules coincide with  $\Phi^+(C)$? We refer to \cite{GLS2020} for related consideration on rigid locally free $H$-modules.

Assume that $\mathbb{K}$ is  algebraically closed. By  \cite[Theorem~2]{Kac80}, for any $\alpha\in \Phi^+(\Delta)$, there is an indecomposable $\mathbb{K}\Delta$-module $M$ with $\underline{\rm dim}(M)=\alpha$. Combining the surjectivity of $\boldsymbol{f}\colon \Phi^+(\Delta)\rightarrow \Phi^+(C)$ and Theorem~\ref{thm:categorify}, we infer the following fact: for each $\beta\in \Phi^+(C)$,  there is an indecomposable $\tau$-locally free $H$-module $X$ with $\underline{\rm rank}(X)=\beta$. This fact supports an affirmative answer to the above open question.
\end{rem}

We illustrate Proposition~\ref{prop:Dynkin} with an explicit example.

\begin{exm}
Let $\mathbb{K}$ be a field of characteristic two, and let $\Delta$ be the following quiver of type $A_3$.
\[\xymatrix{ 2 \ar[r]^-{\alpha}& 1   & 2' \ar[l]_-{\alpha'}}\]
The Auslander-Reiten quiver $\Gamma_{\mathbb{K}\Delta}$ is as follows.
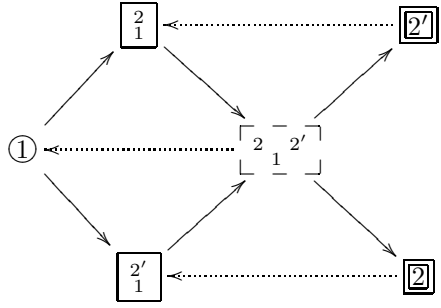
\[\xymatrix{
&  \begin{xy}*+{\begin{smallmatrix} 2\\ 1\end{smallmatrix}}*\frm{-}\end{xy} \ar[rd] && \ar@{.>}[ll] \begin{xy} *+{2'}*\frm{=}\end{xy}\\
\begin{xy} *+{1}*\frm{o}\end{xy}  \ar[ru] \ar[rd] && \ar@{.>}[ll]\begin{xy}*+{\begin{smallmatrix}2 & & 2'\\ & 1& \end{smallmatrix}}*\frm{--}\end{xy}  \ar[ru] \ar[rd] \\
& \begin{xy}*+{\begin{smallmatrix} 2'\\ 1\end{smallmatrix}}*\frm{-}\end{xy} \ar[ru] && \ar@{.>}[ll]\begin{xy} *+{2}*\frm{=}\end{xy}
}\]
Here, the dotted arrows denote the Auslander-Reiten translation. We visualize each module using its radical layers, and represent composition factors by their corresponding vertices.

 Let  $G=\{1_G, \sigma\}$ be a cyclic group of order two, and let $\sigma$ act on $\Delta$ by interchanging $\alpha$ and $\alpha'$. The associated Cartan triple $(C, D, \Omega)$ is of type $B_2$ and  given as follows:
$$C=\begin{pmatrix}2 & -1 \\
                   -2 & 2\end{pmatrix}, D={\rm diag}(2, 1), \mbox{ and } \Omega=\{(1,2)\}.$$
The algebra $H=H(C, D, \Omega)$ is given by the following quiver
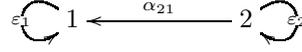
\[
\xymatrix{
1 \ar@(ul,dl)[]|{\varepsilon_1}  && \ar[ll]_-{\alpha_{21}} 2 \ar@(ur,dr)[]|{\varepsilon_2}
}\]
subject to relations $\varepsilon_1^2=0=\varepsilon_2$. In practice, one simply deletes the loop $\varepsilon_2$.

The Auslander-Reiten quiver $\Gamma_H$ is as follows; see \cite[Subsection~13.6]{GLS}.
\[\xymatrix{
&  &  \begin{xy}*+{\begin{smallmatrix}2 \\ 1 \\ 1 \end{smallmatrix}}*\frm{-}\end{xy} \ar[rd] && \ar@{.>}[ll] \begin{xy}*+{2}*\frm{=}\end{xy}\\
 &  \begin{xy} *+{\begin{smallmatrix}1 \\ 1 \end{smallmatrix}}*\frm{o}\end{xy} \ar[ur] \ar[rd] && \ar@{.>}[ll]\begin{xy}*+{\begin{smallmatrix} 2 & & 2\\ & 1\\ & 1 \end{smallmatrix}}*\frm{--}\end{xy} \ar[ur] \ar[rd]\\
 1\ar[rd]  \ar[ur] && \ar@{.>}[ll] {\begin{smallmatrix} 1 & & 2\\ & 1 \end{smallmatrix}} \ar[ur] \ar[rd] && \ar@{.>}[ll] {\begin{smallmatrix} 2\\ 1\end{smallmatrix}}\\
 & {\begin{smallmatrix} 2\\ 1\end{smallmatrix}} \ar[ur] && \ar@{.>}[ll]  1\ar[ur]
}\]
Here, the leftmost and rightmost arrows in the bottom are identified. We have framed all the indecomposable $\tau$-locally free $H$-modules. The central three-dimensional $H$-module is locally free, but not $\tau$-locally free.

 We apply Proposition~\ref{prop:Dynkin} to obtain the bijection
 $$\Theta=\Psi\circ (-\#G)\colon \mathbb{K}\Delta\mbox{-{\rm ind}}/G \stackrel{\sim}\longrightarrow H\mbox{-{\rm ind}}^{\tau \mbox{-{\rm lf}}}. $$
The twisting endofunctor on $\mathbb{K}\Delta\mbox{-{\rm mod}}$ with respect to $\sigma$ turns $\Gamma_{\mathbb{K}\Delta}$ upside down. By comparing $\Gamma_{\mathbb{K}\Delta}$ and $\Gamma_H$, we observe that $\Theta$ preserves the frames of the modules, that is, each indecomposable $\mathbb{K}\Delta$-module $M$ and $\Theta(M)$ have the same kind of frames.

By Lemma~\ref{lem:Tr}, the bijection $\Theta$ is compatible with Auslander-Reiten translations. The following observation might be compared with \cite[Theorem~3.8]{RR}: by applying $\Theta$ to the square in $\Gamma_{\mathbb{K}\Delta}$, we infer that, in general, $\Theta$ does not preserve Auslander-Reiten sequences.
\end{exm}

\vskip 10pt

\noindent {\bf Acknowledgements.}\quad This paper is partly written  when Wang  visited University of Stuttgart in 2019; she is grateful to Steffen Koenig for inspiring comments and his hospitality. Chen thanks Jan Schr\"{o}er for a helpful email concerning Remark~\ref{rem:categorify}(2). We thank Bernhard Keller for enlightening discussions. This work is supported by National Natural Science Foundation of China (No.s 11901551 and 11971449) and the Fundamental Research Funds for the Central Universities.

\bibliography{}

\vskip 10pt

 {\footnotesize \noindent Xiao-Wu Chen, Ren Wang\\
 Key Laboratory of Wu Wen-Tsun Mathematics, Chinese Academy of Sciences,\\
 School of Mathematical Sciences, University of Science and Technology of China, Hefei 230026, Anhui, PR China}

\end{document}